\documentclass[12pt,reqno]{amsart}

\usepackage[noadjust]{cite}

\usepackage{hyperref}
\usepackage{orcidlink}

\usepackage{xcolor}

\usepackage{amsmath,amsthm,amsfonts,amssymb,amscd,  multicol, verbatim, enumerate, stackrel}
\usepackage{verbatim}
\usepackage{fullpage}
\input{xy}
\usepackage{xcolor}
\usepackage[english]{babel}
\usepackage{cite}
\usepackage{amsfonts}
\usepackage{latexsym}
\usepackage{amsthm}
\usepackage{amsopn}
\usepackage{amsmath}
\usepackage[all]{xy}

\usepackage{hyperref}

\usepackage{verbatim}
\usepackage{amssymb}
\usepackage{mathrsfs}
\usepackage{float}
\usepackage{fullpage}
\usepackage{amscd}
\usepackage{amscd}
\usepackage{color}

\usepackage{graphicx}
\usepackage{tikz}

\theoremstyle{plain}
\newtheorem*{thm*}{Theorem}
\newtheorem*{con*}{Conjecture}
\newtheorem{thm}{Theorem}[section]
\newtheorem{pro}[thm]{Proposition}

\newtheorem{lem}[thm]{Lemma}
\newtheorem{con}[thm]{Conjecture}

\newtheorem{cor}[thm]{Corollary}

\newtheorem*{str*}{Strategy}

\theoremstyle{definition}
\newtheorem{defn}[thm]{Definition}

\theoremstyle{remark}
\newtheorem{rmk}[thm]{Remark}
\newtheorem{exm}[thm]{Example}

\usepackage[noEnd=true,indLines=false]{algpseudocodex}

\newcommand{\D}{\mathbb{D}}

\title[Flip Combinatorial Invariance]{Flip Combinatorial Invariance and Weyl groups}
\author[Esposito, Marietti, Stella]{Francesco Esposito\, \orcidlink{0000-0003-4627-9193}, Mario Marietti\, \orcidlink{0000-0001-7045-0347}, Salvatore Stella\, \orcidlink{0000-0001-5390-2081}}

\date{}

\address{Francesco Esposito, Dipartimento di Matematica, Universit\`a degli Studi di Padova,
via Trieste 63, 35121 Padova, Italy}
\address{Mario Marietti, Dipartimento  di Ingegneria Industriale e Scienze Matematiche, Universit\`a Politecnica delle Marche, via Brecce Bianche, 60131 Ancona,  Italy}
\address{Salvatore Stella, Dipartimento di Ingegneria e Scienze dell'Informazione e Matematica, Universit\`a degli Studi dell'Aquila, Via Vetoio SNC, 67100 L'Aquila, Italy}

\email{esposito@math.unipd.it}
\email{m.marietti@univpm.it}
\email{salvatore.stella@univaq.it}


\begin{document}

\maketitle

\begin{abstract}
In this work, we investigate the approach via flipclasses to the Combinatorial Invariance Conjecture for Kazhdan--Lusztig polynomials  of all Coxeter groups.
We prove the combinatorial invariance of  Kazhdan--Lusztig $\widetilde{R}$-polynomials of  Weyl groups modulo $q^7$ and of  Kazhdan--Lusztig $\widetilde{R}$-polynomials of type $A$  Weyl groups modulo $q^8$.  As a consequence, we establish the Combinatorial Invariance Conjecture for all intervals up to length 10 in Weyl groups and up to length 12 in type $A$ Weyl groups.
\end{abstract}

\section{Introduction}

Kazhdan--Lusztig polynomials, introduced  by Kazhdan and Lusztig in \cite{K-L},  have contributed to the birth of what is now called geometric representation theory and have opened vast areas of research in representation theory, geometry, and algebraic combinatorics. These polynomials are indexed by two elements  $u$ and $v$ in a Coxeter group $W$, and, a priori, depend on the algebraic structure of $W$. It is a long-standing conjecture that, in fact, these polynomials depend only on the combinatorial structure. This conjecture  is now commonly referred to as the Combinatorial Invariance Conjecture and  has been the focus of active research for the past forty years (we refer the reader to \cite{BreOPAC} for details).

\begin{con}
\label{comb-inv-con}
The Kazhdan--Lusztig polynomial $P_{u,v}(q)$ depends only on the isomorphism type of the interval $[u,v]$ as a partially ordered set under Bruhat order.
\end{con}
\noindent 
The Combinatorial Invariance Conjecture for Kazhdan--Lusztig polynomials is equivalent to the Combinatorial Invariance Conjecture for $\widetilde{R}$-polynomials  (see Conjecture~\ref{comb-inv-con2}). Given $u,v \in W$, the $\widetilde{R}$-polynomial $\widetilde{R}_{u,v}$ counts by length the paths from $u$ to $v$ in the Bruhat graph whose labels are increasing with respect to any fixed reflection ordering (see Theorem~\ref{Dyertilde}).

This challenging problem has been the subject of extensive research but it is still widely open. The conjecture has been established in several special cases, among which: 
 intervals of rank $\leq 4$ (see, e.g., \cite[7.31]{Dyeth}, and \cite[Chapter~5, Exercises~7 and~8]{BB}), intervals of rank $\leq 8$ in type $A$ and $\leq 6$ in types $B$ and $D$ (\cite{Inc}), 
   intervals in type $\tilde{A}_2$ (\cite{BLP}), intervals which are lattices (\cite[7.23]{Dyeth}, \cite{Bre94}),  and intervals starting from the identity  (\cite{BCM1}).
In the very recent paper \cite{BGL}, posted on arXiv after the first version
of this work, the authors prove the combinatorial invariance of the
coefficient of $q^{\ell(v)-\ell(u)-2}$ in the $\widetilde{R}$-polynomial
$\widetilde{R}_{u,v}(q)$ and, as a consequence, show that the
Combinatorial Invariance Conjecture holds for intervals of rank at most $6$
in arbitrary Coxeter groups. Another subsequent paper \cite{CM} approaches Conjecture \ref{comb-inv-con} for intervals of type $A$ via the poset-theoretic notion of special matchings.
Also a parabolic version of  Conjecture \ref{comb-inv-con} holds for intervals starting from the identity (\cite{Mtrans}, \cite{M}).
On the other hand,  certain entirely poset-theoretical generalizations of Conjecture \ref{comb-inv-con} fail (\cite{BCM2}, \cite{MJaco}).

Recently, fresh ideas  to attack the  Combinatorial Invariance Conjecture have been found  in \cite{BBDVW}  with the help of certain machine learning models (see also \cite{DVBBZTTBBJLWHK}). These ideas generated  several new studies (see \cite{B-G}, \cite{B-M}, \cite{EM2}, \cite{G-W}). In particular, in \cite{B-G}, the Combinatorial Invariance Conjecture is proven for elementary intervals of type $A$.

\bigskip
In \cite{EM}, a new approach has been proposed for the Combinatorial Invariance Conjecture for the symmetric group.
The underlying idea is that intervals should be further decomposed into smaller entities, viz. the flipclasses, and that the  Combinatorial Invariance Conjecture for  $\widetilde{R}$-polynomials in fact holds at the deeper level of  flipclasses. 
In this work, we extend this line of research to all Coxeter groups.

Let $W$ be an arbitrary Coxeter group, and denote  its Bruhat graph by $B(W)$. For each $u$ and $v$  in $W$, with $u\leq v$ in Bruhat order, and each $h \in \mathbb N$, we partition the set  of paths from $u$ to $v$ of length $h$ in $B(W)$ into smaller sets, the {\em $h$-flipclasses} from $u$ to $v$, and we introduce a notion of combinatorial isomorphism of flipclasses.
Since the  multiset of the combinatorial isomorphism classes of $h$-flipclasses of an interval depends only on the interval as a poset, 
one can tackle the Combinatorial Invariance Conjecture by investigating to which extent (i.e., for which set $\mathcal F$ of flipclasses in possibly different Coxeter groups) the following assertion holds.

\noindent{\bf Flip Combinatorial Invariance}.
Let $F$ and $F'$  be two combinatorially isomorphic flipclasses belonging to  $\mathcal F$. Then  the number of increasing paths in $F$ and the number of increasing paths in $F'$ do not depend on the chosen reflection orderings, and these two numbers coincide.

Given the set $\mathcal F$ of all flipclasses in a class $\mathcal W$ of Coxeter groups,  the  Flip Combinatorial Invariance for $\mathcal F$ implies  the Combinatorial Invariance Conjecture for the Kazhdan--Lusztig polynomials of the Coxeter groups in $\mathcal W$.

The following theorem collects our main results on flipclasses (see Definitions~\ref{defd} and~\ref{deffin} for the definitions of the flipclass types).

\begin{thm}
\label{colleziona}
$ $ 

\begin{enumerate}
\item[(i)] 
A flipclass of finite type has the same number of increasing paths with respect to any reflection ordering.

\item[(ii)]
\label{eccoci}
 Flip Combinatorial Invariance holds for
\begin{itemize}
\item dihedral flipclasses (of any length);
\item  $h$-flipclasses of Weyl type, for $h\leq 6$;
\item $h$-flipclasses of symmetric type, for $h\leq 7$.
\end{itemize}
\end{enumerate}
\end{thm}
Flip Combinatorial Invariance does not hold for flipclasses in the finite non-Weyl Coxeter groups $H_3$ and $H_4$.

Let us briefly outline the steps of the proof of the statements of Theorem~\ref{colleziona} regarding  flipclasses of Weyl and symmetric types, which  is computer-assisted. First, we use the results in Section~6  to reduce the analysis to a finite number of cases. Then, we find a fine enough invariant of flipclasses, the valence polynomial: checking whether flipclasses are combinatorially isomorphic is computationally heavier than checking equality of valence polynomials. Finally, we use some isomorphisms and anti-isomorphisms of flipclasses to reduce to an amount of cases  that is tractable by  computer. The details on the computations performed by  computer are in Section~8.

The following statements collect the consequences of our results on the Combinatorial Invariance Conjecture.
\begin{cor}
\label{combinatoriainvarianzaintro}
$ $

\begin{enumerate}
\item[(i)]
\label{pressione}
Let $[u_1,v_1]$ and   $[u_2,v_2]$ be two isomorphic intervals (of any length) in two Weyl groups. Then    
 $$\widetilde{R}_{u_1,v_1} \equiv  \widetilde{R}_{u_2,v_2}  \pmod {q^7}.$$
If, moreover, one restricts to Weyl groups of type $A$, then  
$$\widetilde{R}_{u_1,v_1} \equiv  \widetilde{R}_{u_2,v_2}  \pmod {q^8}.$$
\item[(ii)]
The Combinatorial Invariance Conjecture holds for all intervals 
\begin{itemize}
\item up to length 10, for the set of Weyl groups, 
\item
up to length 12, for the set of Weyl groups of type $A$.
\end{itemize}
\end{enumerate}
\end{cor}

Based on our investigation, we believe that the following conjecture might be true. 
\begin{con}
Flip Combinatorial Invariance holds for  flipclasses of Weyl type.
\end{con}

 \medskip
 The paper is organized as follows. In Section~2, we review the preliminary notions that we use in the paper. In Section~3, we set up the theory of flipclasses in arbitrary Coxeter groups. In Section~4, we completely solve the case  of dihedral flipclasses. In Section~5, we prove that a flipclass in a finite Coxeter group has the same number of increasing paths with respect to any reflection ordering. In Section~6, for each $h$, we show how to reduce the study of  $h$-flipclasses of intervals (of any length) in finite Coxeter groups to a finite number of cases.   In Section~7, we present our results concerning Flip Combinatorial Invariance.    In Section~8, we give the details of the proof of Theorem~\ref{flow}.

\section{Preliminaries}

In this section, we  review some background material. 

For $n\in \mathbb N^+$, we let $[n]=\{1,2,\ldots,n\}$, and  $[0,n]=\{0,1,\ldots,n\}$.

\subsection{Coxeter groups}
We fix our notation on a Coxeter group $(W,S)$ in the following list:
\smallskip 

$
\begin{array}{@{\hskip-1.3pt}l@{\qquad}l}
e &  \textrm{the identity of $W$}, 
\\
\ell  &  \textrm{the length function of $W$ with respect to $S$},
\\
w_0 &  \textrm{the longest   element of $W$},
\\
T &  \textrm{the set of  reflections of $W$, i.e. } \{ w s w ^{-1} : w \in W, \; s \in S \}, 
\\
\Phi &  \textrm{the set of roots of $W$},
\\
\Phi^+ &  \textrm{the set of positive roots of $W$},
\\
t_{\alpha}  &  \textrm{the reflection corresponding to the root $\alpha$},
\\
\alpha_t  &  \textrm{the positive root corresponding to the reflection $t$},
\\
\leq & \textrm{the Bruhat order on $W$ (as well as usual order on $\mathbb R$)},
\\
\textrm{$[u,v]$} &  \textrm{the (Bruhat) interval generated by $u,v\in W$, i.e. } \{ w \in W \, : \; u \leq w \leq v \},
\\
B(W)  &  \textrm{the  Bruhat graph  of $W$},
\\
B(X)  &  \textrm{the (directed) graph induced on the subset $X$ by the Bruhat graph  $B(W)$},
\\
P_h(u,v) & \textrm{the set of paths of length $h$ from $u$ to $v$ in $B(W)$}.
\end{array}$
\bigskip

We refer the reader to  \cite [Chapter 4]{BB} for the basic facts regarding the standard geometric representation of $W$. Recall that the sets of reflections $T$ and of positive roots $\Phi^+$ are in natural bijection: for short, we  denote  the reflection corresponding to the root $\alpha$ by $t_{\alpha}$ and the root corresponding to the reflection $t$ by $\alpha_t$.

The \emph{Bruhat graph} of $W$, denoted $B(W)$, is the edge-labelled directed graph whose vertex set is $W$ and where  $u \xrightarrow{t} v$ if and only if $\ell(u)<\ell(v)$ and $vu^{-1}=t \in T$. The \emph{Bruhat order} is the partial order on $W$ where $u \leq v$ whenever there is a (directed) path from $u$ to $v$ in $B(W)$.

 Let $u,v \in W$, $u \leq v$, and  $\Gamma= (u=x_0 \stackrel{t_1}{\longrightarrow} x_1 \stackrel{t_2}{\longrightarrow}  \cdots \stackrel{t_{h-1}}{\longrightarrow} x_{h-1}\stackrel{t_h}{\longrightarrow}  x_h=v)$ be a path from $u$ to $v$ in $B(W)$. We define
  $h$ to be the {\em length} of $\Gamma$, denoted $\ell(\Gamma)$. 
 For short, we say  $x_i \in \Gamma$, for all $i\in [0,h]$, as well as   $x_i \stackrel{t_{i+1}}{\longrightarrow} x_{i+1} \in \Gamma$,  for all $i\in [0,h-1]$.

The next result follows from (in fact, is proven in the proof of) \cite[Lemma 3.1]{DyeComp1}.
\begin{pro}
\label{stessopiano}
Let $u,v \in W$. Let $\Gamma_1 = \big( u\stackrel{t_1}{\longrightarrow}  x_1\stackrel{r_1}{\longrightarrow}  v \big)$ and $\Gamma_2 = \big( u\stackrel{t_2}{\longrightarrow}  x_2\stackrel{r_2}{\longrightarrow}  v\big)$ be two paths of length 2 from $u$ to $v$. Then $\operatorname{span}\{\alpha_{t_1}, \alpha_{r_1}\}=\operatorname{span}\{\alpha_{t_2}, \alpha_{r_2}\}=\operatorname{span}\{\alpha_{t_1}, \alpha_{t_2}\}=\operatorname{span}\{\alpha_{r_1}, \alpha_{r_2}\}$.
\end{pro}

\subsection{Reflection subgroups and reflection orderings}
The fundamental results on reflection subgroups have been established by Dyer  in several papers (e.g., \cite{DyeJAlg} and  \cite{DyeComp1}).
Let $W'$ be any reflection subgroup of a Coxeter group $W$, i.e., $W'= \langle W' \cap T\rangle$. Then $W'$ is itself a Coxeter group with the set $\{ t \in T : \{r\in T \cap W' : \ell(rt)<\ell(t) \}=\{t\}  \}$, denoted $\chi(W')$, as a set of Coxeter generators  (see \cite{DyeJAlg}). 

Recall (see \cite[\S 2]{DyeComp}) that a {\em reflection ordering} on $(W,S)$ is a total order $\preceq$ on $T$ such that if $(W',\chi(W'))$
is a dihedral reflection subgroup of $W$ (i.e. $\chi(W')$
has cardinality $2$) then either
\[
a \preceq aba \preceq ababa \preceq \cdots \preceq babab \preceq bab \preceq b
\]
or 
\[
b \preceq bab \preceq babab \preceq \cdots \preceq ababa \preceq aba \preceq a
\]
where $\{ a,b \}= \chi(W')$ and the expressions appearing above are reduced and exhaust the reflections in $W'$.

As noted in \cite[Remark 2.4]{DyeComp}, given a reflection subsystem $(W',\chi(W'))$ of a Coxeter system $(W,S)$, the restriction of a reflection ordering on $W$ to an order on $W'\cap T$ is a reflection ordering on $W'$. 
By abuse of language, we refer to the Coxeter system $(W',  \chi(W'))$ simply as the Coxeter subgroup $W'$. 

Given a reflection ordering $\preceq $,  a  path $\Gamma= \big(x_0 \stackrel{t_1}{\longrightarrow} x_1 \stackrel{t_2}{\longrightarrow}  \cdots \stackrel{t_{h-1}}{\longrightarrow} x_{h-1}\stackrel{t_h}{\longrightarrow} x_h\big)$ in $B(W)$ is \emph{increasing with respect to $\preceq $} provided that $t_1 \preceq t_2\preceq \cdots \preceq t_h$.

We will use the following result several times (see \cite[Theorem~3.3 and Corollary~3.4]{DyeJAlg} and  \cite[Theorem~1.4]{DyeComp1}).
\begin{thm}
\label{DyeComp1.thm1.4}
Let $W'$ be a reflection subgroup of $W$. Let $C$ be a left coset of $W'$.  Then there is a unique edge labeling preserving isomorphism from the Bruhat graph $B(W')$ of $W'$ to the subgraph $B(C)$ of the Bruhat graph $B(W)$ of $W$. 
\end{thm}

\subsection{Kazhdan--Lusztig polynomials}
The Kazhdan--Lusztig polynomials $\{ P_{u,v}(q) \}_{u,v \in W}$ of $W$  and the Kazhdan--Lusztig $R$-polynomials $\{ R_{u,v}(q) \}_{u,v \in W} $ of $W$ are two  families of polynomials in one variable indexed by pairs of elements of a Coxeter group $W$. Kazhdan--Lusztig polynomials and $R$-polynomials are equivalent in the following sense:  given $u,v \in W$, one can reconstruct  the set $\{ P_{x,y} (q)\}_{x,y \in [u,v]}$ from the set $\{ R_{x,y}(q) \}_{x,y \in [u,v]}$, and vice versa. For the definition of  Kazhdan--Lusztig and $R$-polynomials, we refer the reader to \cite[Chapter~5]{BB}. In this work, we use a third family of polynomials in one variable indexed by pairs of elements of $W$, which are the 
$\widetilde{R}$-polynomials. 
The  $\widetilde{R}$-polynomial $\widetilde{R}_{u,v}(q)$  is the unique polynomial with natural coefficients satisfying 
$$
R_{u,v}(q)= q^{\frac{\ell(v)-\ell(u)}{2}}\widetilde{R}_{u,v}(q^{\frac{1}{2}}-q^{-\frac{1}{2}}),
$$
(see, e.g., \cite[Proposition 5.3.1]{BB}).
The $\widetilde{R}$-polynomials can also be defined through the  following combinatorial interpretation of  Dyer (see \cite{DyeComp} and also \cite[Theorem~5.3.4]{BB}). 
\begin{thm}
\label{Dyertilde}
Fix a reflection ordering  $\preceq $.  For $u,v \in W$, we have
$$\widetilde{R}_{u,v}(q)=\sum q^{\ell(\Gamma)},$$
where the sum is over all increasing paths $\Gamma$ from $u$ to $v$.
\end{thm}

A classical result of Dyer (see  \cite[Proposition~3.3]{DyeComp1}) asserts that, given $u,v\in W$, the isomorphism type of the directed graph $B([u,v])$ is determined by the isomorphism type  of the  Bruhat interval $[u,v]$ as a poset.  Hence, the Combinatorial Invariance Conjecture is equivalent to the following one.

\begin{con}
\label{comb-inv-con2}
The  $\widetilde{R}$-polynomial $\widetilde{R}_{u,v}(q)$  depends only on the isomorphism type of the graph $B([u,v])$.
\end{con}

\subsection{Time-support graph and time-support poset.}
\label{definizioni}
In this subsection, we recall  the definitions of some concepts introduced in \cite{EM}.

Given a directed graph  $G$, we denote the vertex set and the edge set of $G$ by $V(G)$ and $E(G)$, respectively. Given $h \in \mathbb N$ and $u,v \in V(G)$, we let $P^G_h(u,v)$ denote  the set of paths of length $h$  in $G$ from $u$ to $v$.

\begin{defn}
Let $G$ be an edge-labelled directed graph, $u,v \in V(G)$, and $h \in \mathbb N$.  Let $F$ be a subset of $P^G_h(u,v)$.

The {\em labelled time-support graph of $F$}, denoted $LTS_F$, is the edge-labelled  directed graph such that
\begin{enumerate}
\item $V(LTS_F)= \{ (a,i) \in V(G) \times [0,h]: \text{ $ \exists \big(u=x_0 {\longrightarrow}  \cdots {\longrightarrow} x_{h}=v\big)$ in $F$ with $x_i = a$} \}$ 
\item $E(LTS_F)= \{ (a,i) \stackrel{t}{\longrightarrow}  (b, i+1) : 
\\\text{ $ \exists \big(u=x_0 {\longrightarrow}  \cdots {\longrightarrow} x_{h}=v\big)$ in $F$  with $x_i = a$, $x_{i+1} = b$ and containing $x_i  \stackrel{t}{\longrightarrow}  x_{i+1}$}\}$.
\end{enumerate}
The {\em time-support graph of $F$}, denoted $TS_F$, is the  directed graph obtained by $LTS_F$ by forgetting the labels of the edges.

For $i \in \{0,1, \ldots, h\}$, we say that the path $\big(u=x_0 {\longrightarrow} x_1 {\longrightarrow} \cdots {\longrightarrow} x_{h}=v\big)$ passes through vertex $x_i$ at time $i$ and through the edge $x_i {\longrightarrow} x_{i+1}$ between time $i$ and $i+1$.
\end{defn}
As an example, consider the graph $G$ depicted in Figure~\ref{espopancia0}, on the left. Edges  point upwards and can have arbitrary labels. Let $F$ be the set of all paths from $u$ to $v$ of length 4. Then  the time-support graph $TS_F$ is depicted on the right.

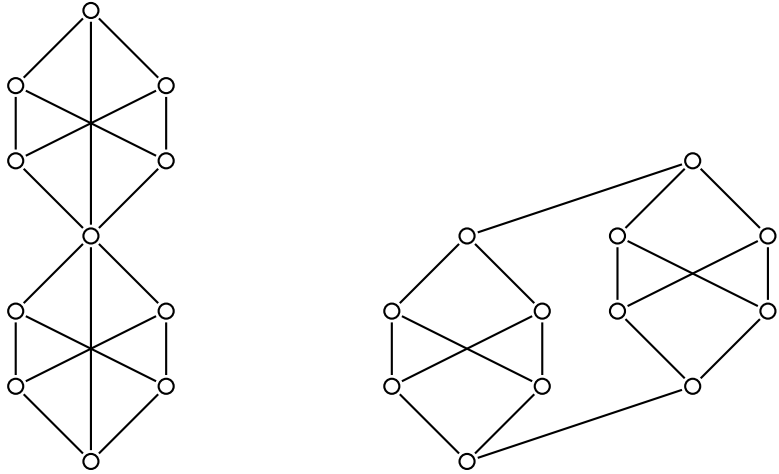
\begin{figure}[h] \label{espomateria}
\centering
\begin{tikzpicture}

     \draw (-3,0) node[fill=white, thick, draw=black, inner sep=0,outer sep=0.5mm, minimum size=2mm, circle](u) {};

    \draw (-4,1) node[fill=white, thick, draw=black, inner sep=0,outer sep=0.5mm, minimum size=2mm, circle](a1) {};
    \draw (-2,1) node[fill=white, thick, draw=black, inner sep=0,outer sep=0.5mm, minimum size=2mm, circle](a2) {};
       
    \draw (-4,2) node[fill=white, thick, draw=black, inner sep=0,outer sep=0.5mm, minimum size=2mm, circle](b1) {};
    \draw (-2,2) node[fill=white, thick, draw=black, inner sep=0,outer sep=0.5mm, minimum size=2mm, circle](b2) {};
        
    \draw (-3,3) node[fill=white, thick, draw=black, inner sep=0,outer sep=0.5mm, minimum size=2mm, circle](c1) {};
        
    \draw (-4,4) node[fill=white, thick, draw=black, inner sep=0,outer sep=0.5mm, minimum size=2mm, circle](d1) {};
    \draw (-2,4) node[fill=white, thick, draw=black, inner sep=0,outer sep=0.5mm, minimum size=2mm, circle](d2) {};
           
    \draw (-4,5) node[fill=white, thick, draw=black, inner sep=0,outer sep=0.5mm, minimum size=2mm, circle](e1) {};
    \draw (-2,5) node[fill=white, thick, draw=black, inner sep=0,outer sep=0.5mm, minimum size=2mm, circle](e2) {};
        
    \draw (-3,6) node[fill=white, thick, draw=black, inner sep=0,outer sep=0.5mm, minimum size=2mm, circle](f1) {};

    \draw (2,0) node[fill=white, thick, draw=black, inner sep=0,outer sep=0.5mm, minimum size=2mm, circle](uu) {};
        
    \draw (1,1) node[fill=white, thick, draw=black, inner sep=0,outer sep=0.5mm, minimum size=2mm, circle](aa1) {};
    \draw (3,1) node[fill=white, thick, draw=black, inner sep=0,outer sep=0.5mm, minimum size=2mm, circle](aa2) {};
           
    \draw (1,2) node[fill=white, thick, draw=black, inner sep=0,outer sep=0.5mm, minimum size=2mm, circle](bb1) {};
    \draw (3,2) node[fill=white, thick, draw=black, inner sep=0,outer sep=0.5mm, minimum size=2mm, circle](bb2) {};
        
    \draw (2,3) node[fill=white, thick, draw=black, inner sep=0,outer sep=0.5mm, minimum size=2mm, circle](cc1) {};
        
    \draw (5,1) node[fill=white, thick, draw=black, inner sep=0,outer sep=0.5mm, minimum size=2mm, circle](cc2) {};
        
    \draw (4,2) node[fill=white, thick, draw=black, inner sep=0,outer sep=0.5mm, minimum size=2mm, circle](dd1) {};
    \draw (6,2) node[fill=white, thick, draw=black, inner sep=0,outer sep=0.5mm, minimum size=2mm, circle](dd2) {};
           
    \draw (4,3) node[fill=white, thick, draw=black, inner sep=0,outer sep=0.5mm, minimum size=2mm, circle](ee1) {};
    \draw (6,3) node[fill=white, thick, draw=black, inner sep=0,outer sep=0.5mm, minimum size=2mm, circle](ee2) {};
        
    \draw (5,4) node[fill=white, thick, draw=black, inner sep=0,outer sep=0.5mm, minimum size=2mm, circle](ff1) {};

    \draw[thick] (u)--(a1);

    \draw[thick] (u)--(a2);

    \draw[thick] (a1)--(b1);

    \draw[thick] (a1)--(b2);

    \draw[thick] (a2)--(b1);

    \draw[thick] (a2)--(b2);

    \draw[thick] (b1)--(c1);

    \draw[thick] (b2)--(c1);

    \draw[thick] (c1)--(d1);

    \draw[thick] (c1)--(d2);

    \draw[thick] (d1)--(e1);

    \draw[thick] (d1)--(e2);

    \draw[thick] (d2)--(e1);

    \draw[thick] (d2)--(e2);

    \draw[thick] (e1)--(f1);

    \draw[thick] (e2)--(f1);

    \draw[thick] (u)--(c1);

    \draw[thick] (f1)--(c1);

    \draw[thick] (uu)--(aa1);

    \draw[thick] (uu)--(aa2);

    \draw[thick] (aa1)--(bb1);

    \draw[thick] (aa1)--(bb2);

    \draw[thick] (aa2)--(bb1);

    \draw[thick] (aa2)--(bb2);

    \draw[thick] (bb1)--(cc1);

    \draw[thick] (bb2)--(cc1);

    \draw[thick] (cc2)--(dd1);

    \draw[thick] (cc2)--(dd2);

    \draw[thick] (dd1)--(ee1);

    \draw[thick] (dd1)--(ee2);

    \draw[thick] (dd2)--(ee1);

    \draw[thick] (dd2)--(ee2);

    \draw[thick] (ee1)--(ff1);

    \draw[thick] (ee2)--(ff1);

    \draw[thick] (uu)--(cc2);

    \draw[thick] (ff1)--(cc1);

\end{tikzpicture}
    \caption{Time-support graph}
    \label{espopancia0}
\end{figure}

The time-support graph  is the Hasse diagram of a poset, that we call  \emph{time-support poset}. The time-support poset of a subset $F$ of $P^G_h(u,v)$, denoted  $TSP_F$, is graded of rank $h$ (i.e., every maximal chain has the same
length $h$), and has a unique minimal element and a unique maximal element. The two notions of time-support graph and time-support poset are equivalent.

\section{Flipclasses for arbitrary Coxeter groups}

In this section, we show how one can define flipclasses of arbitrary Coxeter groups. It is  not clear how to extend the definition of flip outside type $A$ (see \cite{EM}) since there might be more than two paths of length 2 joining the same end-points. 

Throughout this section, except when otherwise stated,  $W$ is any Coxeter group.  Given $h \in \mathbb N$ and $u,v \in  W$, we denote the set of $h$-paths from $u$ to $v$ in the Bruhat graph $B(W)$ simply by $P_{h}(u,v)$ (i.e., $P_{h}(u,v) = P^{B(W)}_{h}(u,v))$. 

\begin{lem}
\label{mari}
Let $u,v \in W$. 
\begin{enumerate}
\item $P_2(u,v)$ contains an even number of paths.
\item Let $X=\{x\in [u,v]\setminus\{u,v\} : \text{ there exists $\Gamma$ in $P_2(u,v)$ containing $x$} \}$.  Then there exists a unique fixed-point free involution  $f :X \to X$ such that, given $x_1,x_2 \in X$, with $\ell(x_1) < \ell(x_2)$ and $x_2\neq f(x_1)$, we have $\ell(f(x_1)) < \ell(f(x_2))$.  

\end{enumerate} 
\end{lem}
\begin{proof}
By \cite[Lemma~3.1]{DyeComp1}, the subgroup $W'$ generated by all reflections labeling some edge in some path in  $P_2(u,v)$ is a dihedral reflection subgroup of $W$. Let $C$ be the coset of $W'$ containing $u$ (and hence also $X$ and $v$). By Theorem~\ref{DyeComp1.thm1.4}, there is a unique edge labeling preserving isomorphism $\phi$ from the graph $B(W')$ to the graph $B(C)$. Notice that $\phi$  is order preserving (while $\phi^{-1}$ might not be).
 
For short, let $u'$, $v'$, $X'$ correspond to $u$, $v$, $X$.

The first statement holds for $P_2(u,v)$ since it holds for $P_2(u',v')$.

The second statement clearly holds for $X'$. Indeed, such unique involution is the unique fixed-point free map $f':X' \to X'$ such that $\ell(x) = \ell(f'(x))$. Let $f= \phi \circ f'  \circ  \phi^{-1}$. Notice that, in general,  $\ell(x) \neq \ell(f(x))$. 

Let $x_1,x_2 \in X$, with $\ell(x_1) < \ell(x_2)$ and $x_2\neq f(x_1)$. Let $x_1'$ and $x_2'$ correspond to $x_1$ and $x_2$. Since $x_2'\neq f'(x_1')$, we have $\ell(x'_1) < \ell(x'_2)$. Then $\ell(f'(x'_1)) < \ell(f'(x'_2))$: hence  $f'(x'_1) < f'(x'_2)$ since $W'$ is a dihedral group. Being $\phi$ order preserving, we have $f(x_1) < f(x_2)$ and so $\ell(f(x_1)) < \ell(f(x_2))$.

The uniqueness follows by the uniqueness of $f'$.
\end{proof}

\begin{rmk}
Let $u,v\in W$. Let $W'$ be the dihedral reflection subgroup  generated by all reflections labeling some edge in some path in  $P_2(u,v)$. Then, by the classification of reflection subgroups of Weyl groups which may be traced back to \cite{BoSi} and \cite{Dy} (see \cite[Section 2]{DyLe} for a clear, concise and self contained treatment), we note the following.  

\begin{enumerate}
\item 
If $W$ is a simply laced Weyl group,  then  $|P_{2}(u,v)|  \in \{0,2\}$, for all $u,v\in W$.
\item
If $W$ is a Weyl group of type $B$ or $F_4$, then  $|P_{2}(u,v)| \in\{0,2,4\}$, for all $u,v\in W$. More precisely, suppose $|P_{2}(u,v)| \neq 0$  and let  $ \big( u\stackrel{t_\alpha}{\longrightarrow}  x\stackrel{t_\beta}{\longrightarrow}  v \big) \in P_{2}(u,v)$. Then 
\begin{equation*}
P_{2}(u,v)= \left\{ \begin{array}{ll}
4, & \mbox{if $\alpha \perp \beta$ and $\Phi \cap \operatorname{span}\{\alpha, \beta\}$ is of type $B_2$}; \\
2, & \mbox{otherwise.} 
\end{array} \right. 
\end{equation*}
Indeed, suppose $|P_{2}(u,v)| \neq 0$ and let  $V=\operatorname{span}\{\alpha_t : t \text{ is a label in a path in $P_{2}(u,v)$}\}$. We have $\dim V=2$, by Proposition~\ref{stessopiano}.  Let $\Phi'=\Phi \cap V$. Then $\Phi'$ is a crystallographic root system of rank 2 of type either $A_1\times A_1$, or $A_2$, or $B_2$.  
By  Theorem~\ref{DyeComp1.thm1.4}, there is an edge labeling preserving isomorphism  between the graph induced by the paths of length 2 from $u$ to $v$ and a subgraph of the Bruhat graph of the Weyl group of $\Phi'$. Hence, the assertion follows.
\item
For each even number $n$, there exist a Coxeter group $W$ and $u,v\in W$ with $|P_{2}(u,v)| =n$ (for instance, the dihedral Coxeter group of order $2n$).
\end{enumerate}
\end{rmk}
The following definitions generalize \cite[Definition~3.7]{EM}.
\begin{defn}
\label{flipflip}
The \emph{flip} of a path $ \big(u{\longrightarrow} x {\longrightarrow} v)$ is the path  $\big(u{\longrightarrow} f(x) {\longrightarrow} v)$, where $f$ is the involution of Lemma~\ref{mari}.  

Let $u,v \in W$  and $h \in \mathbb N$. 
\begin{itemize}
\item
The {\em $i$-th flip} is the map $f_i: P_{h}(u,v) \to P_{h}(u,v)$ that sends a path $\Gamma = \big(u=x_0 {\longrightarrow} x_1 {\longrightarrow} \cdots {\longrightarrow} x_{h}=v\big)$ to the path obtained from $\Gamma$ by substituting the subpath $\big( x_{i-1} {\longrightarrow} x_i {\longrightarrow}  x_{i+1}\big)$ with its flip, for $i\in [h-1]$.
\item 
An \emph{$h$-flipclass} of $W$ from $u$ to $v$ is a minimal subset of $P_{h}(u,v)$ closed under flips.
\end{itemize}
\end{defn}

\begin{rmk}
$ $

\begin{enumerate}
\item The definition of $i$-th flip, and hence the definition of a flipclass, depends only on the isomorphism type of the Bruhat graph (i.e., it is independent of the labels of the edges of the Bruhat graph).
\item 
A different notion of flip for arbitrary Coxeter groups is introduced in \cite{Bla}.  However, it does not fit our needs because it depends on the labels of the edges of the Bruhat graph. 
\item  Since $\ell(t)$ is odd for each $t\in T$,  there are no $h$-paths from $u$ to $v$, and a fortiori  no $h$-flipclasses of $W$ from $u$ to $v$, unless $h \equiv \ell(v)-\ell(u) \pmod 2$.
\end{enumerate}
\end{rmk}

\begin{exm}
In order to clarify the definition of flip, we give the following example. Let $W$ be the Weyl group of type $B_3$ with Coxeter generators $s,r,p$ such that $sr$, $rp$, and $sp$ have order $4$, $3$ and $2$, respectively. Let $u=e$, i.e., the identity element,  and $v=prsrps$. Suppose that we want to perform the flip of a length-2 path from $u$ to $v$ and all we know is the isomorphism class of the Bruhat graph $B([u,v])$ of the interval $[u,v]$ (we do not know the elements labeling the vertices and the reflections labeling the edges). We look at $P_{2}(u,v)$ and find that it consists of 4 paths. Even though we do not know that these paths are the following 4 paths:
\begin{itemize}
\item[-]  $\Gamma = \big( u\stackrel{s}{\longrightarrow}  x_1= s\stackrel{prsrp}{\longrightarrow}  v \big)$, 
\item[-]$\Gamma' = \big( u\stackrel{prp}{\longrightarrow} x_2=  prp \stackrel{psrsp}{\longrightarrow}  v \big)$,
\item[-] $\Delta = \big( u\stackrel{sprps}{\longrightarrow}  x_3= sprps\stackrel{prp}{\longrightarrow}  v \big)$, and 
\item[-] $\Delta' = \big( u\stackrel{prsrp}{\longrightarrow}  x_4= prsrp\stackrel{s}{\longrightarrow}  v \big)$,
\end{itemize}
we know the lengths of their middle elements $x_i$, $i\in[4]$ (see Figure~\ref{espopancia2}). We order the paths compatibly with these lengths (for instance, $\Delta $ could be placed either just before or just after $\Delta'$) and we match the paths according to this order.  So the flip of $\Gamma$ is $\Gamma'$ (and vice versa) and the flip of $\Delta$ is $\Delta '$ (and vice versa).
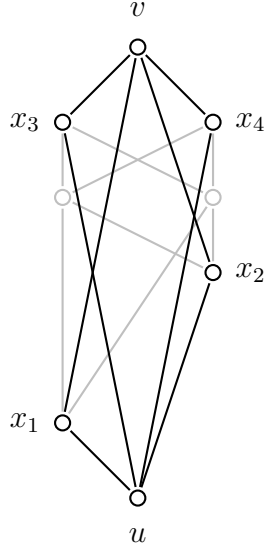
\begin{figure}[h] \label{espomateria2}
\centering
    \begin{tikzpicture}
      \draw (-3,0) node[fill=white, thick, draw=black, inner sep=0,outer sep=0.5mm, minimum size=2mm, circle](u) {};

      \draw (-4,1) node[fill=white, thick, draw=black, inner sep=0,outer sep=0.5mm, minimum size=2mm, circle](x1) {};
      \draw (-2,3) node[fill=white, thick, draw=black, inner sep=0,outer sep=0.5mm, minimum size=2mm, circle](x2) {};
      
      \draw (-4,4) node[fill=white, thick, draw=lightgray, inner sep=0,outer sep=0.5mm, minimum size=2mm, circle](a1) {};
      \draw (-2,4) node[fill=white, thick, draw=lightgray, inner sep=0,outer sep=0.5mm, minimum size=2mm, circle](a2) {};
             
      \draw (-4,5) node[fill=white, thick, draw=black, inner sep=0,outer sep=0.5mm, minimum size=2mm, circle](x3) {};
      \draw (-2,5) node[fill=white, thick, draw=black, inner sep=0,outer sep=0.5mm, minimum size=2mm, circle](x4) {};

      \draw (-3,6) node[fill=white, thick, draw=black, inner sep=0,outer sep=0.5mm, minimum size=2mm, circle](v) {};

    \draw[thick] (u)--(x1);

    \draw[thick] (u)--(x2);

    \draw[thick, draw=lightgray] (a1)--(x1);

    \draw[thick, draw=lightgray] (a1)--(x2);

    \draw[thick, draw=lightgray] (a2)--(x1);

    \draw[thick, draw=lightgray](a2)--(x2);

    \draw[thick, draw=lightgray](a1)--(x3);

    \draw[thick, draw=lightgray](a1)--(x4);

    \draw[thick, draw=lightgray](a2)--(x3);

    \draw[thick, draw=lightgray](a2)--(x4);

    \draw[thick] (u)--(x3);

    \draw[thick] (u)--(x4);

    \draw[thick] (v)--(x1);

    \draw[thick] (v)--(x2);

    \draw[thick] (v)--(x3);

    \draw[thick] (v)--(x4);

    \node (u') at (-3,-0.5) {$u$};
    \node (x1') at (-4.5,1) {$x_1$};   
    \node (x2') at (-1.5,3) {$x_2$};
       
    \node (x3') at (-4.5,5) {$x_3$};
    \node (x4') at (-1.5,5) {$x_4$};

    \node (v') at (-3,6.5) {$v$};   

    \end{tikzpicture}
    \caption{Example of flips}
    \label{espopancia2}
\end{figure}

\end{exm}

\begin{defn}
Let $F$ be a flipclass of an arbitrary Coxeter group. We let 
\begin{eqnarray*}
E(F) &=&\{w\in W : \textrm{ there exists $\Gamma$ in $F$ with $w\in \Gamma$} \},\\
T(F) &=&\{t\in T : \textrm{ there exists $\Gamma$ in $F$ with an edge labelled by $t$} \},
\end{eqnarray*}
and we denote by $W(F)$ the reflection subgroup generated by $T(F)$. 
\end{defn}

\begin{rmk}
\label{nonsempreuguali}
As opposed to the case of flipclasses in Weyl groups of type $A$ (see \cite[Lemma~6.1]{EM}), the subgroup $W(F)$ might be larger than the reflection subgroup generated by the reflections of a single path in $F$ (as already happens for the $2$-flipclass consisting of  $\big( e\stackrel{s}{\longrightarrow}  s\stackrel{rsr}{\longrightarrow}  w_0 \big)$ and  $\big( e\stackrel{r}{\longrightarrow} r \stackrel{srs}{\longrightarrow}  w_0 \big)$  in the Coxeter system $(W,\{s,r\})$  of type $B_2$).
\end{rmk}

\begin{defn}
Let $W$ and $W'$ be two Coxeter groups,  and $h \in \mathbb N$. Let $F$ be an $h$-flipclass  of $W$ and $F'$ be an $h$-flipclass  of $W'$. 
\begin{itemize}
\item  We say that $F$ and $F'$ are \emph{isomorphic} (respectively, \emph{anti-isomorphic}) provided that there exists a pair $(\varepsilon,\tau)$ such that:
\begin{itemize}
\item $\varepsilon$ is a bijection from $E(F)$ to $E(F')$ and $\tau$ is a bijection from $T(F)$ to $T(F')$ such that   
$$\big(x_0 \stackrel{t_1}{\longrightarrow} x_1 \stackrel{t_2}{\longrightarrow}  \cdots \stackrel{t_{h-1}}{\longrightarrow} x_{h-1}\stackrel{t_h}{\longrightarrow} x_h\big) \in F$$ if and only if  
$$\big(\varepsilon(x_0) \stackrel{\tau(t_1)}{\longrightarrow} \varepsilon(x_1) \stackrel{\tau(t_2)}{\longrightarrow}  \cdots \stackrel{\tau(t_{h-1})}{\longrightarrow} \varepsilon(x_{h-1})\stackrel{\tau(t_h)}{\longrightarrow} \varepsilon(x_h)\big) \in F'$$ 
(respectively, $$\big(\varepsilon(x_h) \stackrel{\tau(t_h)}{\longrightarrow} \varepsilon(x_{h-1}) \stackrel{\tau(t_{h-1})}{\longrightarrow}  \cdots \stackrel{\tau(t_{2})}{\longrightarrow} \varepsilon(x_{1})\stackrel{\tau(t_1)}{\longrightarrow} \varepsilon(x_0)\big) \in F');$$
\item  the induced  bijection $\bar{\varepsilon}$ from $F$ to $F'$ is flip-preserving, i.e., $\bar{\varepsilon}f_i(\Gamma)= f_{i} \bar{\varepsilon}(\Gamma)$ (respectively, $\bar{\varepsilon}f_i(\Gamma)= f_{h-i} \bar{\varepsilon}(\Gamma)$) for all $\Gamma \in F$;
\item there exist a reflection ordering of $W$ and a reflection ordering of $W'$ w.r.t. which $\tau$ is order preserving  (respectively, order reversing).
\end{itemize}

\item  We say that $F$ and $F'$ are \emph{combinatorially isomorphic}  (respectively, \emph{combinatorially anti-isomorphic}) provided that there exists a bijection  $\varepsilon$  from $E(F)$ to $E(F')$ such that:
 $$\big(x_0{\longrightarrow} x_1 {\longrightarrow}  \cdots {\longrightarrow} x_{h-1}{\longrightarrow} x_h\big) \in F$$ if and only if  
$$\big(\varepsilon(x_0) {\longrightarrow} \varepsilon(x_1) {\longrightarrow}  \cdots {\longrightarrow} \varepsilon(x_{h-1}){\longrightarrow} \varepsilon(x_h)\big) \in F'$$ 
(respectively, $$\big(\varepsilon(x_h) {\longrightarrow} \varepsilon(x_{h-1}) {\longrightarrow}  \cdots {\longrightarrow} \varepsilon(x_{1}){\longrightarrow} \varepsilon(x_0)\big) \in F')$$
and  the induced  bijection $\bar{\varepsilon}$ from $F$ to $F'$ is flip-preserving, i.e., $\bar{\varepsilon}f_i(\Gamma)= f_{i} \bar{\varepsilon}(\Gamma)$ (respectively, $\bar{\varepsilon}f_i(\Gamma)= f_{h-i} \bar{\varepsilon}(\Gamma)$) for all $\Gamma \in F$.
\end{itemize}
\end{defn}
Given two flipclasses $F$ and $F'$, we write $F \cong F'$ (respectively,  $F\stackrel{\text{\tiny{cb}}}\cong F'$) if they are isomorphic (respectively, combinatorially isomorphic).

Clearly, isomorphic flipclasses are combinatorially isomorphic. Furthermore, combinatorially isomorphic flipclasses have  isomorphic time-support graphs and  time-support posets. If two flipclasses are combinatorially anti-isomorphic, then their time-support graphs, as well as their  time-support  posets, are dual to each other.

\begin{rmk}
\label{costante}
Let $BI$ be the set of all Bruhat intervals in all Coxeter groups, i.e., the set of the posets $P$ such that there exist a Coxeter group $W$ and $u,v\in W$ with $P=[u,v]$.
Let $f$ be a function having $BI$ as  domain. We say that $f$ is a {\em combinatorial invariant of Bruhat intervals}, or just a combinatorial invariant, provided that $f$ is constant on the isomorphism classes (in the category of posets). 
For each $h$ in $\mathbb N$, associating with an interval $[u,v]$ the multiset of combinatorial isomorphism classes of the $h$-flipclasses  from $u$ to $v$, as well as  the multiset of the isomorphism classes of their  time-support graphs or poset,  gives combinatorial invariants of Bruhat intervals. 
\end{rmk}

\begin{exm}
\label{nottefonda2} 
Let $W$ be the Weyl group of type $A_3$ with Coxeter generators $s_1,s_2,s_3$ such that  $s_1s_2$, $s_2s_3$, $s_1s_3$ have order $3$, $3$ and $2$, respectively. Let $u=e$, i.e., the identity element,  and $v=s_1s_2s_3s_2s_1$. 
The interval $[u,v]$ has one $5$-flipclass, two $3$-flipclasses (both combinatorially isomorphic to the dihedral $3$-flipclass, see Definition~\ref{defd}) and one $1$-flipclass. These flipclasses are depicted in Figure~\ref{4_flipclass}, where elements are denoted with the standard one-line notation.

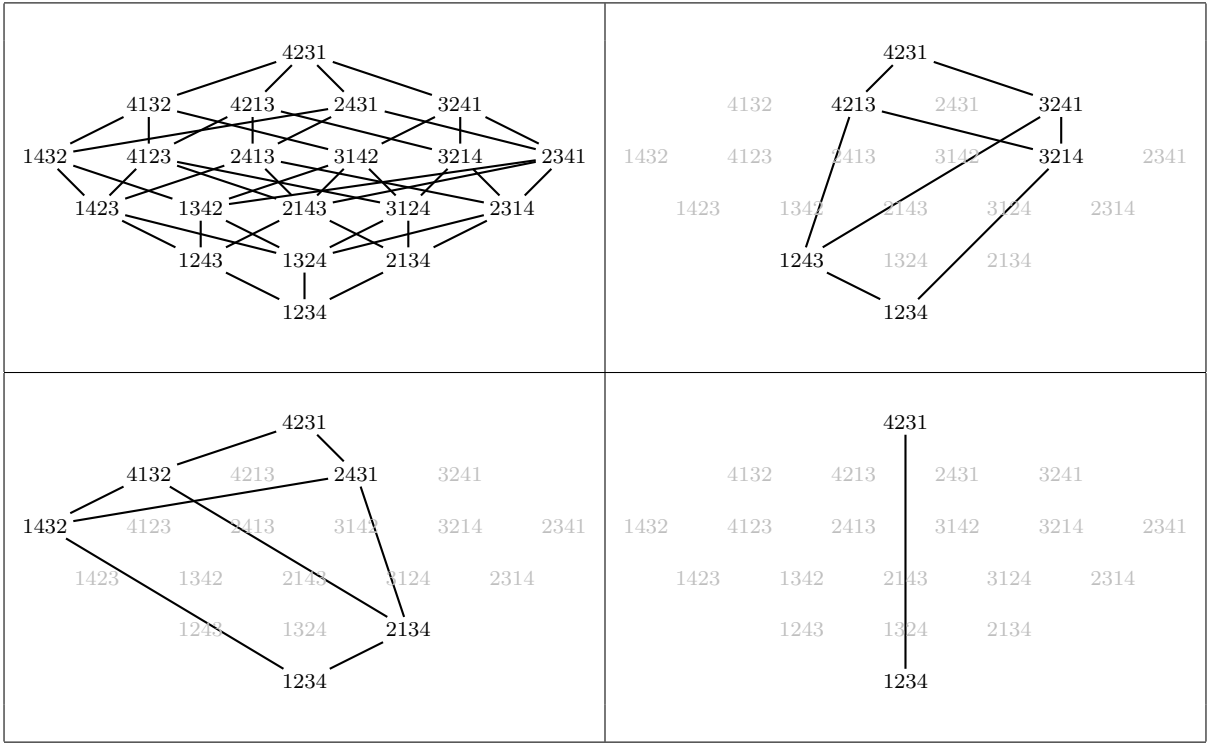
\begin{figure}
\begin{center}
\begin{tabular}{|c|c|}
\hline
& \\
 \begin{tikzpicture}[scale=.69,every node/.style={inner sep = 2pt}]
    \node (u) at (0,0) {{\tiny $1234$}};
    \node (x1) at (-2,1) {{\tiny $1243$}};
    \node (x2) at (0,1) {{\tiny $1324$}};
    \node (x3) at (2,1) {{\tiny $2134$}};
    \node (y1) at (-4,2) {{\tiny $1423$}};
    \node (y2) at (-2,2) {{\tiny $1342$}};
    \node (y3) at (0,2) {{\tiny $2143$}};
    \node (y4) at (2,2) {{\tiny $3124$}};
    \node (y5) at (4,2) {{\tiny $2314$}};
    \node (z1) at (-5,3) {{\tiny $1432$}};
    \node (z2) at (-3,3) {{\tiny $4123$}};
    \node (z3) at (-1,3) {{\tiny $2413$}};
    \node (z4) at (1,3) {{\tiny $3142$}};
    \node (z5) at (3,3) {{\tiny $3214$}};
    \node (z6) at (5,3) {{\tiny $2341$}};
    \node (w1) at (-3,4) {{\tiny $4132$}};
    \node (w2) at (-1,4) {{\tiny $4213$}};
    \node (w3) at (1,4) {{\tiny $2431$}};
    \node (w4) at (3,4) {{\tiny $3241$}};
    \node (v) at (0, 5) {{\tiny $4231$}};

    \draw[thick] (u)--(x1);
    \draw[thick] (u)--(x2);
    \draw[thick] (u)--(x3);
    \draw[thick] (x1)--(y1);
    \draw[thick] (x1)--(y2);
    \draw[thick] (x1)--(y3);
    \draw[thick] (x2)--(y1);
    \draw[thick] (x2)--(y2);
    \draw[thick] (x2)--(y4);
    \draw[thick] (x2)--(y5);
    \draw[thick] (x3)--(y3);
    \draw[thick] (x3)--(y4);
    \draw[thick] (x3)--(y5);
    \draw[thick] (y1)--(z1);
    \draw[thick] (y1)--(z2);
    \draw[thick] (y1)--(z3);
    \draw[thick] (y2)--(z1);
    \draw[thick] (y2)--(z4);
    \draw[thick] (y2)--(z6);
    \draw[thick] (y3)--(z2);
    \draw[thick] (y3)--(z3);
    \draw[thick] (y3)--(z4);
    \draw[thick] (y3)--(z6);
    \draw[thick] (y4)--(z2);
    \draw[thick] (y4)--(z4);
    \draw[thick] (y4)--(z5);
    \draw[thick] (y5)--(z3);
    \draw[thick] (y5)--(z5);
    \draw[thick] (y5)--(z6);
    \draw[thick] (z1)--(w1);
    \draw[thick] (z1)--(w3);
    \draw[thick] (z2)--(w1);
    \draw[thick] (z2)--(w2);
    \draw[thick] (z3)--(w2);
    \draw[thick] (z3)--(w3);
    \draw[thick] (z4)--(w1);
    \draw[thick] (z4)--(w4);
    \draw[thick] (z5)--(w2);
    \draw[thick] (z5)--(w4);
    \draw[thick] (z6)--(w3);
    \draw[thick] (z6)--(w4);
    \draw[thick] (w1)--(v);
    \draw[thick] (w2)--(v);
    \draw[thick] (w3)--(v);
    \draw[thick] (w4)--(v);
    
    \end{tikzpicture} &

\begin{tikzpicture}[scale=.69,every node/.style={inner sep = 2pt}]
    \node (u) at (0,0) {{\tiny $1234$}};
    \node (x1) at (-2,1) {{\tiny $1243$}};
    \node (z5) at (3,3) {{\tiny $3214$}};
    \node (w2) at (-1,4) {{\tiny $4213$}};
    \node (w4) at (3,4) {{\tiny $3241$}};
    \node (v) at (0, 5) {{\tiny $4231$}};

    \draw[thick] (u)--(x1);
    \draw[thick] (u)--(z5);
    \draw[thick] (x1)--(w2);
    \draw[thick] (x1)--(w4);
    \draw[thick] (z5)--(w2);
    \draw[thick] (z5)--(w4);
    \draw[thick] (w2)--(v);
    \draw[thick] (w4)--(v);

    \node[lightgray] (x2) at (0,1) {{\tiny $1324$}};
    \node[lightgray] (x3) at (2,1) {{\tiny $2134$}};
    \node[lightgray] (y1) at (-4,2) {{\tiny $1423$}};
    \node[lightgray] (y2) at (-2,2) {{\tiny $1342$}};
    \node[lightgray] (y3) at (0,2) {{\tiny $2143$}};
    \node[lightgray] (y4) at (2,2) {{\tiny $3124$}};
    \node[lightgray] (y5) at (4,2) {{\tiny $2314$}};
    \node[lightgray] (z1) at (-5,3) {{\tiny $1432$}};
    \node[lightgray] (z2) at (-3,3) {{\tiny $4123$}};
    \node[lightgray] (z3) at (-1,3) {{\tiny $2413$}};
    \node[lightgray] (z4) at (1,3) {{\tiny $3142$}};
    \node[lightgray] (z6) at (5,3) {{\tiny $2341$}};
    \node[lightgray] (w1) at (-3,4) {{\tiny $4132$}};
    \node[lightgray] (w3) at (1,4) {{\tiny $2431$}};
    
    \end{tikzpicture}
    \\ 
& \\
\hline
& \\
\begin{tikzpicture}[scale=.69,every node/.style={inner sep = 2pt}]
    \node (u) at (0,0) {{\tiny $1234$}};
    \node (x3) at (2,1) {{\tiny $2134$}};
    \node (z1) at (-5,3) {{\tiny $1432$}};
    \node (w1) at (-3,4) {{\tiny $4132$}};
    \node (w3) at (1,4) {{\tiny $2431$}};
    \node (v) at (0, 5) {{\tiny $4231$}};

    \draw[thick] (u)--(x3);
    \draw[thick] (u)--(z1);
    \draw[thick] (x3)--(w1);
    \draw[thick] (x3)--(w3);
    \draw[thick] (z1)--(w1);
    \draw[thick] (z1)--(w3);
    \draw[thick] (w1)--(v);
    \draw[thick] (w3)--(v);

    \node[lightgray] (x1) at (-2,1) {{\tiny $1243$}};
    \node[lightgray] (x2) at (0,1) {{\tiny $1324$}};
    \node[lightgray] (y1) at (-4,2) {{\tiny $1423$}};
    \node[lightgray] (y2) at (-2,2) {{\tiny $1342$}};
    \node[lightgray] (y3) at (0,2) {{\tiny $2143$}};
    \node[lightgray] (y4) at (2,2) {{\tiny $3124$}};
    \node[lightgray] (y5) at (4,2) {{\tiny $2314$}};
    \node[lightgray] (z2) at (-3,3) {{\tiny $4123$}};
    \node[lightgray] (z3) at (-1,3) {{\tiny $2413$}};
    \node[lightgray] (z4) at (1,3) {{\tiny $3142$}};
    \node[lightgray] (z5) at (3,3) {{\tiny $3214$}};
    \node[lightgray] (z6) at (5,3) {{\tiny $2341$}};
    \node[lightgray] (w2) at (-1,4) {{\tiny $4213$}};
    \node[lightgray] (w4) at (3,4) {{\tiny $3241$}};

    \end{tikzpicture}
    & 
\begin{tikzpicture}[scale=.69,every node/.style={inner sep = 2pt}]
    \node (u) at (0,0) {{\tiny $1234$}};
    \node (v) at (0, 5) {{\tiny $4231$}};

    \draw[thick] (u)--(v);
    
    \node[lightgray] (x1) at (-2,1) {{\tiny $1243$}};
    \node[lightgray] (x2) at (0,1) {{\tiny $1324$}};
    \node[lightgray] (x3) at (2,1) {{\tiny $2134$}};
    \node[lightgray] (y1) at (-4,2) {{\tiny $1423$}};
    \node[lightgray] (y2) at (-2,2) {{\tiny $1342$}};
    \node[lightgray] (y3) at (0,2) {{\tiny $2143$}};
    \node[lightgray] (y4) at (2,2) {{\tiny $3124$}};
    \node[lightgray] (y5) at (4,2) {{\tiny $2314$}};
    \node[lightgray] (z1) at (-5,3) {{\tiny $1432$}};
    \node[lightgray] (z2) at (-3,3) {{\tiny $4123$}};
    \node[lightgray] (z3) at (-1,3) {{\tiny $2413$}};
    \node[lightgray] (z4) at (1,3) {{\tiny $3142$}};
    \node[lightgray] (z5) at (3,3) {{\tiny $3214$}};
    \node[lightgray] (z6) at (5,3) {{\tiny $2341$}};
    \node[lightgray] (w1) at (-3,4) {{\tiny $4132$}};
    \node[lightgray] (w2) at (-1,4) {{\tiny $4213$}};
    \node[lightgray] (w3) at (1,4) {{\tiny $2431$}};
    \node[lightgray] (w4) at (3,4) {{\tiny $3241$}};

    \end{tikzpicture}
\\
& \\
\hline
\end{tabular}
\end{center}
\caption{The flipclasses of the interval $[1234,4231]$.}
\label{4_flipclass}
\end{figure}

\end{exm}

The following result is among the tools that we use to reduce the computer calculations to a tractable amount.
\begin{lem}
  \label{lem:(anti)-iso}
Let $W$ be a finite Coxeter group,  and $u,v \in W$. Let $F$ be an $h$-flipclass from $u$ to $v$. 
Then
\begin{enumerate}
\item $(\varepsilon : x\mapsto xw_0, \tau : t\mapsto t)$ is an anti-isomorphism between $F$ and a flipclass from $vw_0$ to $uw_0$;
\item $(\varepsilon : x\mapsto w_0x,   \tau :  t\mapsto w_0tw_0)$ is  an anti-isomorphism  between $F$ and a flipclass from $w_0v$ to $w_0u$;
\item $(\varepsilon : x\mapsto w_0xw_0,  \tau :  t\mapsto w_0tw_0)$ is an isomorphism  between $F$ and a flipclass from $w_0uw_0$ to $w_0vw_0$.
 \end{enumerate}
\end{lem}
\begin{proof}
The multiplication by $w_0$ on either side yields an anti-isomorphism of the Bruhat graph inducing the identity map or the conjugation by $w_0$ on the labels. This operation is flip-preserving, hence sends a flipclass to a flipclass.

Suppose we have fixed a reflection ordering $\preceq$. For  $(x\mapsto xw_0, t\mapsto t)$, we consider the opposite reflection ordering, and for  $( x\mapsto w_0x,  t\mapsto w_0tw_0)$,  we consider the opposite of the reflection ordering obtained by conjugation by $w_0$. In both cases,   $\tau$ is order preserving.

The last statement follows by the first two considering  the reflection ordering obtained by conjugation by $w_0$. 
\end{proof}

\bigskip
The following two lemmas are used in  Section~\ref{finitensess}.
\begin{lem}
\label{FeF'}
Let $F$ be a flipclass of an arbitrary Coxeter group $W$. Let $\overline{W}$ be a reflection subgroup containing $T(F)$. 
Then there exist a flipclass $F'$ of $\overline{W }$ and an edge-labeling preserving isomorphism from $F$ to  $F'$.
\end{lem}
\begin{proof}
Recall that, given a subset $X$ of a Coxeter group $W$, we denote by $B(X)$ the directed graph induced on  $X$ by the Bruhat graph  $B(W)$.  Let $C$ be the left coset of $\overline{W}$ containing each path in $F$. 
The assertion follows by Theorem~\ref{DyeComp1.thm1.4}, which ensures that there is an edge labeling preserving isomorphism  between the graph $B(C)$ and the graph $B(\overline{W})$. 
\end{proof}
Given a flipclass $F$, we denote the set of roots $\{\alpha_t: t\in T(F)\}$ by $\Phi(F)$. Compare the next result with Remark~\ref{nonsempreuguali}.
\begin{lem}
\label{sempreuguali}
\label{phirango<h}
Let $F$ be an $h$-flipclass in an arbitrary Coxeter group. Then 
$$ \operatorname{span}\Phi(F)= \operatorname{span} \{\alpha_t: t \emph{ labels an edge in $\Gamma$} \},$$
for any $\Gamma$ in $F$. 
In particular,   $\dim(\operatorname{span}\Phi(F))\leq h $.
\end{lem}

\begin{proof}
The assertion is proved if we show that, given $\Gamma,\Gamma ' \in F$, we have 
$$\operatorname{span} \{\alpha_t: t \emph{ labels an edge in $\Gamma$} \} = \operatorname{span} \{\alpha_t: t \emph{ labels an edge in $\Gamma'$} \}.$$
By transitivity, it is enough to prove the last statement when $\Gamma$ and  $\Gamma'$ differ by a flip, where the assertion readily follows by  Proposition~\ref{stessopiano}.
\end{proof}

Next we study how to decompose flipclasses.
\begin{defn}
We say that a flipclass $F$ is \emph{reducible} if there exists a nontrivial partition $T(F)=T_1\coprod T_2$ such that $t_1t_2=t_2t_1$ for all $t_1\in T_1$ and  $t_2\in T_2$. Otherwise, we say that $F$ is \emph{irreducible}.
\end{defn}

Let $G_1$ and $G_2$ be two edge-labelled directed graphs. Recall that the cartesian product $G_1 \times G_2$ of  $G_1$ and $G_2$  is the edge-labelled directed graph whose vertex set is $V(G_1) \times V(G_2)$ and edge set is 
$$\{ (x_1, x_2) \stackrel{t}{\longrightarrow} (y_1,y_2) : \text{ $x_1 = y_1$ and $x_2\stackrel{t}{\longrightarrow}  y_2 \in E(G_2)$, or  $x_1 \stackrel{t}{\longrightarrow}  y_1 \in E(G_1)$ and $x_2 = y_2$}\}.$$ 
Given a path $\Gamma_1$ in $G_1$ from $u_1$ to $v_1$ of length $h_1$ and  a  path $\Gamma_2$ in $G_2$ from $u_2$ to $v_2$ of length $h_2$, a shuffle of $\Gamma_1$ and $\Gamma_2$ is a path in $G_1 \times G_2$ from $(u_1,u_2)$ to $(v_1,v_2)$ of length $h_1+h_2$ such that if it contains  $(x_1,x_2) \stackrel{t}{\longrightarrow} (y_1,y_2)$ then either  $x_1 = y_1\in \Gamma_1$ and $x_2\stackrel{t}{\longrightarrow}  y_2 \in \Gamma_2$, or  $x_1 \stackrel{t}{\longrightarrow}  y_1 \in \Gamma_1$ and $x_2 = y_2\in \Gamma_2$. 
Given a set of paths $F_1$ in $G_1$ and  a set of paths $F_2$ in $G_2$, we let $F_1 \ast F_2$ denote the set of paths in $G_1 \times G_2$ obtained by all possible shuffles of paths of $F_1$ and $F_2$. We call $F_1 \ast F_2$  the \emph{product} of $F_1$ and $ F_2$.

\begin{rmk}
\label{moltimolti}
Given a set of paths $F_1$ in an edge-labelled directed graph $G_1$ and  a set of paths $F_2$ in an edge-labelled directed graph $G_2$, the time-support graph $TS(F_1\ast F_2)$ is isomorphic to $TS(F_1) \times TS( F_2)$ and the time-support poset $TSP(F_1\ast F_2)$ is isomorphic to $TSP(F_1) \times TSP( F_2)$. Notice that also the labels are well-behaved.
\end{rmk}

\begin{pro}
\label{virgolaid}
Let $F$ be a reducible $h$-flipclass with $T(F)=T_1\coprod T_2$ a nontrivial partition such that $t_1t_2=t_2t_1$ for all $t_1\in T_1$ and  $t_2\in T_2$. Let $W_1=\langle T_1 \rangle$ and $W_2=\langle T_2 \rangle$. Then, there exist $h_1,h_2\in \mathbb N_{>0}$ with $h_1+h_2 =h$ and an $h_1$-flipclass $F_1$ of $W_1$ with $T(F_1)=T_1$ and an $h_2$-flipclass $F_2$ of $W_2$ with $T(F_2)=T_2$ such that $F$ is isomorphic to the flipclass $F_1\ast F_2$ of the product $W_1\times W_2$.
\end{pro}
\begin{proof}
Observe that $W(F)$ is the subgroup $W _1\times W_2$ and recall that  a reflection ordering restricts to a reflection ordering on any reflection subgroup. By Lemma~\ref{FeF'},  there exist a flipclass $F'$ of $W _1\times W_2$ and an edge-labeling preserving isomorphism from $F$ to  $F'$. It is straightforward  to verify that an  $h$-flipclass of $W_1\times W_2$ is of the form $F_1\ast F_2$ for two flipclasses $F_1$ and $F_2$ satisfying the properties of the assertion. 
\end{proof}

\begin{lem}
\label{irriso}
Irreducibility is a property of isomorphism classes.
\end{lem}
\begin{proof}
Let $F$ and $F'$ be two flipclasses and $(\varepsilon, \tau)$ an isomorphism from $F$ to $F'$. Suppose that $F$ is reducible and  $T(F)=T_1\coprod T_2$ is a  nontrivial partition  such that $t_1t_2=t_2t_1$ for all $t_1\in T_1$ and  $t_2\in T_2$. We claim that the nontrivial partition $T(F')=\tau(T_1) \coprod \tau(T_2)$ shows that $F'$ is reducible. In order to prove the claim, we prove that $\tau(t_1)$ and $\tau(t_2)$ commute for all $t_1\in T_1$ and $t_2\in T_2$. By Proposition~\ref{virgolaid},  there is a path $\Gamma$ in $F$ such that $t_1$ and $t_2$ are consecutive labels and are exchanged by a flip, say the $i$-th flip. Hence the path in $F'$ corresponding to $\Gamma$,  together with its $i$-th flip (which is the path corresponding to the $i$-th flip of $\Gamma$), shows that $\tau(t_1) $ and $\tau(t_2)$ commute.
\end{proof}
\begin{rmk}
Irreducibility is not a property of combinatorial isomorphism classes. For example, let $(W,\{s_1,s_2\})$ be the Weyl group of type $A_2$ and  $(W',\{r_1,r_2\})$  be the Weyl group of type $A_1\times A_1$. Then $P_2(e,s_1s_2) $ and $P_2(e,r_1r_2) $ are two combinatorially isomorphic (but not isomorphic) flipclasses and the former is irreducible while the latter is not.
\end{rmk}

\begin{cor}
  \label{cor:unique factorization}
An isomorphism class of flipclasses has a unique factorization as a product of irreducible isomorphism classes. 
\end{cor}
\begin{proof}
Existence follows by Lemma~\ref{irriso} and iterating Proposition~\ref{virgolaid}. Uniqueness follows by the fact the irreducible factors are read from $T(F)$.
\end{proof}
The next result extends to all Coxeter groups a result that is known to hold for Weyl groups of  type $A$ (see \cite[Lemma~3.2]{B-G} or \cite[Lemma~2.4]{EM}). It will be further generalized by Proposition \ref{oradipranzo}.

\begin{lem}
\label{espomariancona24}
Let $u,v \in W$ and let $\preceq$ be a reflection ordering of $W$. Let $\big(u  \stackrel{t_1} {\longrightarrow} x  \stackrel{t_2}{\longrightarrow} v)$ and $\big(u \stackrel{r_1}{\longrightarrow} y  \stackrel{r_2}{\longrightarrow} v)$ be flips of each other. Then one of the two  paths is increasing and the other is decreasing. Moreover, either $\{t_1,r_2\} \preceq \{t_2,r_1\}$, or  $\{t_2,r_1\} \preceq \{t_1,r_2\}$. 
\end{lem}
\begin{proof}
The subgroup $W'$ generated by $t_1$, $t_2$, $r_1$, and $r_2$ is a dihedral reflection subgroup of $W$ by \cite[Lemma~3.1]{DyeComp1}. Let $C$ be the  coset of $W'$ containing $u,x,y,v$. By Theorem~\ref{DyeComp1.thm1.4}, there is an edge labeling preserving isomorphism  between the graph $B(C)$ and the graph $B(W')$. 
The restriction of a reflection ordering on $W$ to an order on $W'\cap T$ is a reflection ordering on $W'$. 
Hence it is enough to show the claim for a dihedral Coxeter group $W$.

Let $(W,S)$ be a dihedral Coxeter system with $S=\{s,t\}$ and  $m(s,t)=d$. Then $|W| = 2d$, and any element $w\in W$, apart from the identity element $e$ and the longest element $w_0$, has a single reduced expression,
which is either of the form $sts\ldots$ or of the form $tst\ldots$  (both with $\ell(w)$ letters). 
We say that $w$ is of type $s$ or of type $t$ depending on whether it has a reduced expression which starts with $s$ or with $t$. As mentioned, $e$ is the only element of $W$ which is neither of type $s$ nor of type $t$, and $w_0$ is the only element of $W$ which is both of type $s$ and of type $t$. In particular, for a fixed length between $1$ and $d-1$, there are exactly two elements of $W$ of that length, one of type $s$ and the other of type $t$. Furthermore, $w$ is a reflection if and only if it has odd length. Observe also that $W$ admits exactly two reflection orderings:
$$
s\preceq sts \preceq ststs \preceq \cdots \preceq (st)^{d-1}s
$$
and
$$
t \preceq  tst \preceq  tstst \preceq  \cdots \preceq  (ts)^{d-1}t,
$$
where all words are of odd length, going from $1$ to $2d-1$. These two reflection orderings are of course just one the reverse of the other. 

The configuration in the Bruhat graph $B(W)$ of $W$ is as follows:
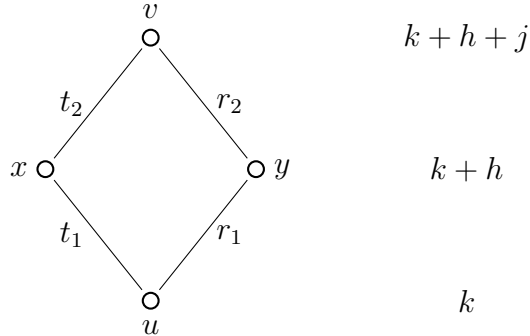
\begin{center}
\begin{tikzpicture}[scale =.7]
\node (u) at (0,-2.5) {};
    
\node (a1) at (-2,0) {};   
\node (a2) at (2,0) {};

\node (b1) at (0,2.5) {};

\node (u') at (0,-2.5 - 0.5) {$u$};
    
\node (a1') at (-2 - 0.5,0) {$x$};   
\node (a2') at (2 + 0.5,0) {$y$};

\node (b1') at (0,2.5 + 0.5) {$v$};

\node (t1) at (-1-0.5,-1.25) {$t_1$};
    
\node (r1) at (1+0.5 ,-1.25) {$r_1$};   
\node (t2) at (-1-0.5,1.25) {$t_2$};

\node (r2) at (1+0.5,1.25) {$r_2$};

\node (lu) at (6,-2.5) {$ k$};
    
\node (lx) at (6,0) {$k+h$};   
\node (lv) at (6,2.5) {$k+h+j$};

\path[-]
    (u) edge (a1)
    (u) edge (a2)
    
    (a1) edge (b1)
    (a2) edge (b1);

\draw[fill=white, thick]  (u) circle (1.5mm);

\draw[fill=white, thick]  (a1) circle (1.5mm);

\draw[fill=white, thick]  (a2) circle (1.5mm);

\draw[fill=white, thick]  (b1) circle (1.5mm);
\end{tikzpicture}
\end{center}
where $0 \leq \ell(u) = k < \ell(x) = \ell(y) = k+h < \ell(v) = k+h+j \leq d$, 
with $h = \ell(t_1) = \ell(r_1)$ and $j= \ell(t_2) = \ell(r_2)$ positive odd integers. 

We suppose, without loss of generality, that $u$ is of type $s$ ($u=e$ works the same). 
Since $x$ and $y$ both have length $k+h$, one among them is of type $t$ and the other is of type $s$. 
Up to renaming,  say $x$ is of type $t$ and $y$ is of type $s$.
Being $x=t_1 u$ and $y= r_1 u$, one has 
$t_1 = tst \ldots $ (with $h$ letters) and $r_1 = sts \ldots $ (with $2k+h$ letters) 

Assume now that $v$ is of type $s$. Then one has 
$t_2 = sts \ldots $ (with $j$ letters) and $r_2 = sts \ldots $ (with $2k+2h+j$ letters).
Since 
$t_1 = sts \ldots $ (with $2d-h$ letters), in the list
\[
s, \ sts, \ ststs, \ldots, (st)^{d-1}s
\]
the reflections $r_1, t_2, r_2, t_1$ appear at indices $2k+h, j, 2k+2h+j, 2d-h$,  respectively. By the above written restrictions on the integers $k,h,j,d$, one gets 
$$j, 2k+h < 2k+h+j < 2k+2h+j, 2d-h.$$ This concludes the proof in this case. The case $v$ of type $t$ is treated analogously.
\end{proof}

The following proposition implies that each flipclass has at least one path that is increasing w.r.t. any reflection order.
\begin{pro}\label{greedy}
Fix a reflection ordering   $\preceq$ of $W$. Let $F$ be a flipclass. The  lexicographically first path $\Gamma$ of $F$ is increasing. 
\end{pro}
\begin{proof}
Let $\Gamma= \big(x_0 \stackrel{t_1}{\longrightarrow} x_1 \stackrel{t_2}{\longrightarrow}  \cdots \stackrel{t_{h-1}}{\longrightarrow} x_{h-1}\stackrel{t_h}{\longrightarrow} x_h\big)$. Toward a contradiction, let $i$ be minimal such that $t_i \succ t_{i+1}$. Let $\Gamma'$ be the $i$-th flip of $\Gamma$. By Lemma~\ref{espomariancona24}, the path $\Gamma'$ is lexicographically smaller than $\Gamma$, which is a contradiction.
\end{proof}

\begin{rmk}
\label{osservazione}
As a consequence of Proposition~\ref{greedy}, we obtain the fact that an interval $[u,v]$ of length $h$ (i.e., $\ell(v)-\ell(u)=h$) has a unique $h$-flipclass $F$, since  the polynomial  $\widetilde{R}_{u,v}$ is monic and Theorem~\ref{Dyertilde} holds. The paths in $F$ correspond to the  maximal chains of the interval $[u,v]$. The  time-support poset of $F$ is isomorphic to $[u,v]$.
Hence, the set of isomorphism classes of time-support posets of flipclasses contains all isomorphism classes of Bruhat intervals.  Unlike Bruhat intervals,  time-support posets might fail to be Eulerian (examples may be found already in type $A$). Furthermore, notice that the uniqueness above yields a direct proof of \cite[Proposition~2.1]{CDM} (which avoids resorting to topological properties of the order complex of the intervals). 
\end{rmk}

\section{Dihedral flipclasses}
In this section, we study flipclasses that are isomorphic to a flipclass of a dihedral Coxeter group, i.e., a  (possibly infinite) Coxeter group of rank 2. 

We recall some basic facts about a dihedral Coxeter group $D$ of cardinality $2m$: 
\begin{itemize}
\item[-] the longest element $w_0$ of $D$ has length $m$;
\item[-] for any $l\in [m-1]$, there are exactly two elements of $D$ of length $l$;
\item[-] the reflections of $D$ are exactly its elements of odd length;
\item[-] for $x,y\in D$, there is an oriented edge in $B(D)$ going from $x$ to $y$ if and only if $\ell(y) - \ell(x)$ is a positive odd integer;
\item[-] the flip of a path $(u \rightarrow x \rightarrow v)$ is the path $(u \rightarrow y \rightarrow v)$, where $y$ is the only element such that $y \neq x$ and $\ell(x) = \ell(y)$.
\end{itemize}

We call an interval {\em dihedral} if it is isomorphic to an interval in a dihedral Coxeter group. Recall that, if $F$ is a flipclass,  then $W(F)$  denotes  the reflection subgroup generated by $T(F)$ and  $\Phi(F)$ denotes the set of roots $\{\alpha_t: t\in T(F)\}$.
\begin{lem}
\label{nottefonda}
Let $F$ be a flipclass in a Coxeter group $W$. The following are equivalent.
\begin{enumerate}
\item $F$ is isomorphic to a flipclass of a dihedral group;
\label{1}

\item $F$ is combinatorially isomorphic to a flipclass of a dihedral group; 
\label{2}

\item the time-support poset  $TSP_F$ is isomorphic to a dihedral interval;
\label{3}

\item the time-support poset  $TSP_F$  has two atoms (or, equivalently, two coatoms);
\label{4}

\item $\operatorname{span} \Phi(F)$ has dimension 2;
\label{5}

\item $\langle t_{\alpha} : \alpha \in \Phi \cap \operatorname{span} \Phi(F)\rangle$ is a dihedral Coxeter group;
\label{6}

\item  $W(F)$  is a dihedral Coxeter group.
\label{7}
\end{enumerate}
\end{lem}
\begin{proof}
(\ref{1})$\implies$(\ref{2}). Obvious.

(\ref{2})$\implies$(\ref{3}). Since the time-support poset is an invariant of the combinatorial isomorphism classes, it is enough to show the assertion for a flipclass $F$ of a dihedral Coxeter group $D$, which is straightforward from the description of $B(D)$ and of the flips.

(\ref{3})$\implies$(\ref{4}). Obvious.

(\ref{4})$\implies$(\ref{5}). Let $V=\operatorname{span}\{\alpha_{t_1}, \alpha_{t_2}\}$, where $t_1,t_2$ are the labels of the two bottom edges of the labelled time-support graph of $F$. Since $\alpha_{t_1}$ and $\alpha_{t_2}$ are independent, $\dim V =2$. Let $t$ be the label of an edge from an element $(x,i)$ to an element $(y,i+1)$ in the labelled time-support graph of $F$. If we show $\alpha_t\in V$, the assertion follows since the set of labels in the  labelled time-support graph of $F$ coincides with $T(F)$. To show $\alpha_t\in V$, we use induction on $i$, the case $i=0$ being clear by the definition of $V$. Let $i>0$ and consider a path $\Gamma= \big(x_0 \stackrel{r_1}{\longrightarrow} x_1 \stackrel{r_2}{\longrightarrow}  \cdots \stackrel{r_{h-1}}{\longrightarrow} x_{h-1}\stackrel{r_h}{\longrightarrow} x_h\big)$ in $F$ passing through the edge $x \stackrel{t}{\longrightarrow}y$ between time $i$ and $i+1$ (i.e., $x_i=x$, $x_{i+1}=y$, and $r_{i+1}=t$).  
Let  $ \big(x_{i-1} \stackrel{p_i} {\longrightarrow} x'  \stackrel{q_i}{\longrightarrow} y)$ be the flip of  $ \big(x_{i-1} \stackrel{r_i} {\longrightarrow} x \stackrel{t} {\longrightarrow} y)$.
Then  $r_i$ and $p_i$ satisfy $\alpha_{r_i},\alpha_{p_i} \in V$  by the induction hypothesis. By Proposition~\ref{stessopiano}, we have  $\alpha_t\in V$.

(\ref{5})$\implies$(\ref{6}).
Follows by \cite[Remark~3.2]{DyeComp1}.

(\ref{6})$\implies$(\ref{7}).
Straightforward since a reflection subgroup of a dihedral group is itself a dihedral group.

(\ref{7})$\implies$(\ref{1}).
Follows by Lemma~\ref{FeF'}.
\end{proof}
\begin{rmk}
(\ref{2})$\iff$(\ref{4}) is a generalization of \cite[Proposition~7.25]{Dyeth}. 
\end{rmk}
\begin{defn}
\label{defd}
We call  {\em dihedral}  a flipclass satisfying the equivalent conditions of Lemma~\ref{nottefonda}.
\end{defn}
\begin{lem}
\label{unasola}
 There is only one combinatorial isomorphism type of dihedral $h$-flipclasses, for each $h$.
\end{lem}
\begin{proof}
Straightforward.
\end{proof}
\begin{rmk}
Although there is only one combinatorial isomorphism type of dihedral $h$-flipclasses, there are many   isomorphism types of dihedral $h$-flipclasses.
\end{rmk}

Given a path $\Gamma$ in a dihedral flipclass $F$, we say that $\Delta\in F$ is opposite to $\Gamma$ if it shares with $\Gamma$ only the first and the last element, i.e., $\Delta$ is obtained from $\Gamma$ by performing all flips once (the order in which the flips are performed is irrelevant).  
\begin{pro}
\label{oradipranzo}
Let $F$ be a dihedral $h$-flipclass. 
Fix a reflection ordering $\preceq$. Then  $F$ has exactly one increasing path $\Gamma= \big(x_0 \stackrel{t_1}{\longrightarrow} x_1 {\longrightarrow}  \cdots {\longrightarrow} x_{h-1}\stackrel{t_h}{\longrightarrow} x_h\big) $.  Moreover, the  path $\Delta = \big(x_0 \stackrel{r_1}{\longrightarrow} y_1 {\longrightarrow}  \cdots {\longrightarrow} y_{h-1}\stackrel{r_h}{\longrightarrow} x_h\big) $ opposite to $\Gamma$ is decreasing and satisfies $t_1\preceq r_1$ and $r_h\preceq t_h$. 
\end{pro}
\begin{proof}
By the definition of a dihedral flipclass, we may suppose that $F$ is a flipclass in a dihedral Coxeter group.

By Proposition~\ref{greedy}, every flipclass $F$ from $u$ to $v$ has at least one increasing path. Let us show that there cannot be more than one increasing path in $F$. We proceed by induction on $h$. For $h = 2$, Lemma~\ref{espomariancona24} settles the question. Suppose now $h > 2$. Let $\Gamma= \big(u \stackrel{t_1}{\longrightarrow} x_1 \stackrel{t_2}{\longrightarrow}  \cdots \stackrel{t_{h-1}}{\longrightarrow} x_{h-1}\stackrel{t_h}{\longrightarrow} v\big)$ be an increasing path in $F$. Consider a different path $\Gamma'= \big(u \stackrel{t'_1}{\longrightarrow} x'_1 \stackrel{t'_2}{\longrightarrow}  \cdots \stackrel{t'_{h-1}}{\longrightarrow} x'_{h-1}\stackrel{t'_h}{\longrightarrow} v\big)$ in $F$, and assume it is increasing as well. 

Suppose that  $\Gamma$ and $\Gamma'$ have a vertex in common differing from $u$ and $v$, i.e., $x_i = x'_i$ for some $i\in [1,h-1]$. The paths $\Gamma_{\leq i} = \big(u {\longrightarrow} x_1  {\longrightarrow} \cdots {\longrightarrow} x_i\big)$ and $\Gamma'_{\leq i} = \big(u {\longrightarrow} x'_1  {\longrightarrow}  \cdots {\longrightarrow} x'_i\big)$ are in the same flipclass and have length $i < h$. Thus, by induction one has
$\Gamma_{\leq i} = \Gamma'_{\leq i}$. Analogously, one argues with the truncations 
$\Gamma_{\geq i} = \big(x_i {\longrightarrow}  \cdots {\longrightarrow} v\big)$ and 
$\Gamma'_{\geq i} = \big(x'_i {\longrightarrow}  \cdots {\longrightarrow} v\big)$ to conclude $\Gamma_{\geq i} = \Gamma'_{\geq i}$; thus ultimately $\Gamma = \Gamma'$, which is a contradiction.

This shows that $\Gamma$ and $\Gamma'$  cannot have common vertex other than start and end, i.e., $\Gamma$ and $\Gamma'$ are mutually opposite.
Since $\ell(x_i) = \ell(x'_i)$ for all $i\in [1,h-1]$, flipping 
$\Gamma$ at $i=1$, one gets the path 
$f_1(\Gamma) = \big(u \stackrel{t'_1}{\longrightarrow} x'_1 \stackrel{r_2}{\longrightarrow} x_2 \stackrel{t_3}{\longrightarrow} \cdots \stackrel{t_{h-1}}{\longrightarrow} x_{h-1}\stackrel{t_h}{\longrightarrow} v\big)$. By Lemma~\ref{espomariancona24}, it follows that $f_1(\Gamma)_{\geq 1} = (x'_1 \stackrel{r_2}{\longrightarrow} x_2 \stackrel{t_3}{\longrightarrow}\cdots \stackrel{t_{h-1}}{\longrightarrow} x_{h-1}\stackrel{t_h}{\longrightarrow} v\big)$ is increasing. Furthermore, it is in the flipclass of 
$\Gamma'_{\geq 1}$, which is also increasing. This contradicts the induction hypothesis and proves that there is exactly one increasing path in $F$.

We conclude the proof by showing that the opposite path $\Gamma'$ of the increasing path $\Gamma$ is in fact decreasing. We argue by induction on $h$. For $h=2$, it follows from Lemma~\ref{espomariancona24}. Suppose now $h > 2$. Since $\Gamma'_{\geq 1}$ is opposite to the increasing path $f_1(\Gamma)_{\geq 1}$ of length $h-1$, by induction one concludes that $\Gamma'_{\geq 1}$ is decreasing. Analogously, $\Gamma'_{\leq h-1}$ is opposite to the increasing path $f_{h-1}(\Gamma)_{\leq h-1}$ of length $h-1$, hence by induction one concludes that also $\Gamma'_{\leq h-1}$ is decreasing. These two statements are equivalent to $\Gamma'$ being decreasing. The last assertion about the reflections follows by  \cite[Corollary~2.3]{BI}. 
\end{proof}
\color{black}

\begin{rmk}
Proposition~\ref{oradipranzo} generalizes  Lemma~\ref{espomariancona24}.
\end{rmk}
We now analyze the set of flipclasses in a  dihedral interval.
Recall that a composition of a positive integer $n$ is a sequence $\underline{\alpha} = (\alpha_1, \ldots, \alpha_r)$ of positive integers such that $\alpha_1 + \ldots + \alpha_r = n$. The integer $r$ is the length of the composition $\underline{\alpha}$. We say that a composition 
$\underline{\alpha} = (\alpha_1, \ldots, \alpha_r)$ is odd if all the $\alpha_i$ are odd.

\begin{pro}
\label{intervallidiedrali}
Let $W$ be an arbitrary Coxeter group. Let $[u,v]$ be a dihedral interval in $W$, $d = \ell(v)-\ell(u)$, and $h \in  [d]$. 
\begin{enumerate}
\item 
\label{uno}
$P_h (u,v) \neq \emptyset \iff h \equiv d  \pmod 2$.
\item 
\label{due}
The $h$-flipclasses from $u$ to $v$ are parametrized by odd compositions of $d$ of length $h$. 
\item 
\label{tre}
When $h \equiv d  \pmod 2$, the number of $h$-flipclasses  from $u$ to $v$ is  $\binom{\frac{d+h}{2}-1}{h-1}$.
\end{enumerate}
\end{pro}
\begin{proof}
Let $\Gamma = \big( u \stackrel{t_1}{\longrightarrow} x_1 {\longrightarrow} \ \ldots {\longrightarrow} \ x_{h-1}\stackrel{t_h}{\longrightarrow}  v \big)$ be a path in $P_h(u,v)$. Let 
$\underline{\ell}(\Gamma) = (\ell(t_1), \ldots, \ell(t_h))$. Then $\underline{\ell}(\Gamma) $ is an odd composition of $d$ of length $h$. By the description of flips in dihedral intervals, the composition 
$\underline{\ell}(\Gamma)$ is invariant by flips and  all paths in $P_h(u,v)$ associated with the same odd composition are in the same flipclass.  
Conversely, given an odd composition $\underline{\gamma}$ of $d$ of length $h$, there exists a path $\Gamma$ with $\underline{\ell}(\Gamma)= \underline{\gamma}$, thanks to the structure of dihedral intervals. Clearly, odd compositions of $d$ of length $h$ exists if and only if $h\equiv d \pmod 2$.

In order to count odd compositions of $d$ of length $h$, one observes that 
\[ (\alpha_1, \ldots, \alpha_h) 
\mapsto 
 (\frac{\alpha_1 + 1}{2}, \ldots, \frac{\alpha_h + 1}{2})
\]
establishes a bijection between odd compositions of $d$ of length $h$ and compositions of $\frac{d+h}{2}$ of length $h$, which are  counted by the binomial coefficient  $\binom{\frac{d+h}{2} - 1}{h-1}$. 
\end{proof}

\begin{rmk}
Proposition \ref{intervallidiedrali}, together with Proposition~\ref{oradipranzo} and Theorem~\ref{Dyertilde}, implies that, if $[u,v]$ is a dihedral interval, then $\widetilde{R}_{u,v}(q)=  \sum_{h} \binom{\frac{d+h}{2}-1}{h-1} q^{h}$, where the sum is over all $h\in[0,d]$ satisfying $h \equiv d  \pmod 2$.
\end{rmk}

\section{Number of increasing paths in a flipclass of a finite Coxeter group}

The following result generalizes \cite[Theorem~4.1]{EM} to any finite Coxeter group and gives a direct argument to prove  that an interval has the same number of increasing paths w.r.t. any reflection ordering. It is actually stronger, establishing the result at the level of flipclasses.

\begin{thm}
\label{prec=prec}
Let $W$ be any finite Coxeter group.
A flipclass $F$ of $W$ has the same number of increasing paths w.r.t. any reflection ordering.
\end{thm}
\begin{proof}
Recall that there is a natural one-to-one correspondence between reduced expressions of the longest element $w_0$ and reflection orderings of $W$  (\cite{Dyeth}, Remark 6.16,(iv)). A reduced expression  $s_1 \cdots s_{\ell(w_0)}$  of $w_0$ corresponds to the reflection ordering $t_1 \preceq t_2 \preceq \cdots  \preceq t_{\ell(w_0)}$, where $t_i = s_1 \cdots s_i  \cdots s_1$, for all $i$ in $[\ell(w_0)]$.
Recall also (see \cite{Tit}) that two reduced expressions of the same element may be joined  by a finite sequence of braid moves, where by a  braid move we mean
the replacement of a consecutive subsequence $ss's \cdots$ of $m(s,s')$ letters by the sequence $s'ss' \cdots$ (again of $m(s,s')$ letters), where $m(s,s')$ is the order of the product $ss'$. 

Let $\preceq$ and $\preceq'$ be two reflection orderings. We need to show that the number of paths in $F$ that are increasing w.r.t. $\preceq$ equals the number of paths in $F$ that are increasing w.r.t. $\preceq'$.  Let $s_1 \cdots s_n$ be the reduced expression of $w_0$ associated with $\preceq$.  We may suppose, by transitivity, that the reduced expression associated with  $\preceq'$ is obtained from  $s_1 \cdots s_n$ by performing a single braid move. Let $k$ be such that the  reduced expression associated with $\preceq'$ is obtained by the braid move that involves $s_k \cdots s_{k+m-1}$. For short, let $s = s_k$ and $r=s_{k+1}$ so that $m$ is the order of the product $sr$. Denote by $t_i$ (respectively, $t_i'$) the $i$-th reflection in the order $\preceq$ (respectively, $\preceq'$).  Then $t_i = t_i'$ for $i\notin \{k, k+1, \ldots, k+m-1\}$, while $t'_k = t_{k+m-1}, t'_{k+1} = t_{k+m-2},  \ldots, t'_{k+m-1}=t_k$. 
Note that $\{t_k, \ldots, t_{k+m-1}\} =\{t'_k, \ldots, t'_{k+m-1}\}$ is the set $\sigma T_{\{s,r\}} \sigma^{-1}$ of the conjugates of the reflections of the parabolic dihedral subgroup $W_{\{s,r\}}$ through the element $\sigma =s_1 \cdots s_{k-1}$.

We construct a bijection $\phi$ from the set $I_{\preceq}$ of increasing paths of $F$ w.r.t. $\preceq$ to the  set $I_{\preceq'}$ of increasing paths of $F$ w.r.t. $\preceq'$.

Let $\Gamma \in I_{\preceq}$. If $\Gamma \in I_{\preceq'}$, then set $\phi(\Gamma) = \Gamma$. 
 If $\Gamma \notin I_{\preceq'}$, then $\Gamma$ has a certain number $j$ of consecutive edges with label in $\sigma T_{\{s,r\}} \sigma^{-1}$ with $j\in [2,m]$, and no other edge with label  in $\sigma T_{\{s,r\}} \sigma^{-1}$. This $j$-path is in a dihedral interval: we obtain  $\phi(\Gamma)$ by replacing it with its opposite path in its flipclass. By Proposition~\ref{oradipranzo},   the path $\phi(\Gamma)$ is increasing w.r.t. $\preceq'$.
The map $\phi$ is an involution, hence a bijection.
\end{proof}

Theorem~\ref{prec=prec} allows the following definition.
\begin{defn}
\label{c==c}
Let $F$ be a flipclass in a finite Coxeter group. We define $c(F)$ to be the number of increasing paths in $F$ (for any choice of reflection ordering).
\end{defn}
The previous theorem and definition are extended to flipclasses of finite type in Proposition~\ref{prec=precfinitetype} and Definition~\ref{defcfinitetype}.

\section{Finiteness results for flipclasses of finite type}
\label{finitensess}
The goal of this section is to prove that, for each $h$, the study of $h$-flipclasses of finite type (see Definition~\ref{deffin}) reduces to a finite number of cases. 

\bigskip
We need the general result given in Proposition~\ref{rango<h}, to prove which we use the following linear algebra lemma.

\begin{lem}
\label{algebretta}
Let $V$ be a Euclidean vector space with scalar product $(\cdot ,\cdot)$,  and $\beta_1,\ldots,\beta_r \in V$. Suppose that
\begin{enumerate}
\item $(\beta_i,\beta_j ) \leq 0$, for all $i,j\in [r]$ with $i\neq j$;
\label{<0}

\item there exists a functional $f\in V^{\ast}$ such that $f(\beta_i)>0$ for each $i \in [r]$.
\label{funzionale}
\end{enumerate}
Then  $\beta_1,\ldots,\beta_r$ are linearly independent.
\end{lem}
\begin{proof}
Toward a contradiction, let $r$ be minimal for which the assertion does not hold. Let $\sum_{i=1}^{r}c_i\beta_i=0$ with $c_i\neq 0$ for at least one $i$ in $[r]$. By minimality, $c_i\neq 0$ for each $i$ in $[r]$. By (\ref{funzionale}), the $c_i$ cannot all have same sign. Hence, there exist $x\in V\setminus{0}$ and a partition $[r]= I \cup J$ such that $x=\sum_{i\in I}c_i\beta_i =   \sum_{j\in J}c'_j\beta_j$ where all coefficients of the two sums are positive ($c'_j=-c_j$, for $j\in J$).
Then $$(x,x)= \big(\sum_{i\in I}c_i\beta_i ,   \sum_{j\in J}c'_j\beta_j \big)= \sum_{(i,j)\in I\times J}c_ic'_j(\beta_i,\beta_j)\leq 0,$$
where the inequality follows by (\ref{<0}). This contradicts the positive definiteness.
\end{proof}
Recall that, for a reflection subgroup $W'$, the set  of its Coxeter generators  is denoted by $\chi(W')$.
\begin{pro}
\label{rango<h}
Let $W$ be a  finite Coxeter group. Let $W'$ be a reflection subgroup  generated by the set of reflections $T'$. Let $h$ be the dimension of the subspace generated by $\{\alpha_t:t\in T'\}$. Then the rank of $W'$ as a Coxeter group does not exceed $h$.
\end{pro}
\begin{proof}
Let $(\cdot ,\cdot)$ denote the bilinear form of the standard geometric representation of $W$. Since $W$ is finite, this bilinear form is  positive definite. 

Let $V=\operatorname{span}\{\alpha_t:t\in T'\}$.
We first observe that $\alpha_r\in V$ for every $r$ in $W'\cap T$.
Indeed, each reflection $t\in T'$ fixes $V^\perp$ pointwise, since
$\alpha_t\in V$. Hence the subgroup $W'=\langle T'\rangle$ fixes $V^\perp$ pointwise. Now let $r\in W'\cap T$. Since $r$ fixes
$V^\perp$ pointwise and $r$ is the reflection corresponding to the root
$\alpha_r$, we have $V^\perp\subseteq \alpha_r^\perp$. Therefore $\alpha_r\in V$.

Let $\chi(W')=\{t_1,\ldots,t_r\}$. 
By \cite[Theorem~4.4]{DyeJAlg}, we have $(\alpha_{t_i},\alpha_{t_j})\leq 0$ for all $i,j\in[r]$, $i\neq j$. Since $\alpha_{t_1},\ldots,\alpha_{t_r}$ are positive roots, they satisfy  the  hypotheses of Lemma~\ref{algebretta}. Hence they are independent. Thus, $|\chi(W')|\leq h$.
\end{proof}

\begin{rmk}
When $W$ is an infinite Coxeter group, Proposition~\ref{rango<h} does not hold. Indeed, a curious phenomenon happens, as shown by  Dyer, Hohlweg and Ripoll in \cite{DHR}. Given an  indefinite Coxeter group $W$ (not finite, nor affine) of rank at least 3 and any $m\in \mathbb N$, there exists a reflection subgroup $W'$ of rank at least $m$ that is a universal Coxeter group  (see  Theorem~6.2 of \cite{DHR}).  
\end{rmk}

\bigskip

Recall that, given a flipclass $F$ of an arbitrary Coxeter group, we denote by $W(F)$ the reflection subgroup generated by $T(F)$.

\begin{lem}
\label{domanda}
Let $F$ be an $h$-flipclass in a finite  Coxeter group $W$. Then the rank of $W(F)$ as a Coxeter group does not exceed $h$.
\end{lem}
\begin{proof}
By Lemma \ref{phirango<h} and Proposition \ref{rango<h}.
\end{proof}

\begin{rmk}
Let $F$ be a flipclass, and  $\Phi(F)=\{\alpha_t: t\in T(F)\}$. In general,  $\Phi \cap \operatorname{span}\Phi(F)$ might strictly contain the root subsystem generated by $\Phi(F)$. For example, let $(W,S)$ be of type $B_2$, with $S=\{s,r\}$, and $F$ be the $2$-flipclass 
$\{\big( r\stackrel{s}{\longrightarrow}  sr \stackrel{rsr}{\longrightarrow}  srs \big), \big( r \stackrel{ rsr}{\longrightarrow} rs \stackrel{s}{\longrightarrow}  srs \big)\}$. Then  $\Phi \cap \operatorname{span}\Phi(F)=\Phi$, while  the root subsystem generated by $\Phi(F)$ is a subsystem of type $A_1\times A_1$ (and  $W(F)=\{e,s,rsr,srsr\}$).
\end{rmk}

\bigskip
In the following definition, we introduce three families of flipclasses that are the object of study in what follows.
\begin{defn}
\label{deffin}
We say that a flipclass $F$ of an arbitrary Coxeter group is of
\begin{itemize}
\item \emph{finite type},  if $W(F)$ is a finite Coxeter group;
\item \emph{Weyl type}, if $W(F)$ is  a Weyl group;
\item  \emph{symmetric type}, if $W(F)$ is a product of Weyl  groups of type $A$.
\end{itemize}
\end{defn}

Clearly, a symmetric type flipclass is of Weyl type, and a Weyl type flipclass is of finite type.

\bigskip
The following result and definition extend  Theorem~\ref{prec=prec} and Definition~\ref{c==c} to flipclasses of finite type.
\begin{pro}
\label{prec=precfinitetype}
Let $W$ be an arbitrary Coxeter group.
A flipclass $F$ of $W$ of finite type has the same number of increasing paths w.r.t. any reflection ordering.
\end{pro}
\begin{proof}
Recall that, given a reflection subgroup $W'$ of $W$, the restriction of a reflection ordering on $W$ to an order on $W'\cap T$ is a reflection ordering on $W'$. 
By Lemma~\ref{FeF'}, the flipclass $F$ is isomorphic to a flipclass $F'$ of $W(F)$ through an edge-labeling preserving isomorphism.
If we had two reflection orderings of $W$ w.r.t. which $F$ has different number of increasing paths, we would have two reflection orderings of $W(F)$ w.r.t. which $F'$ has different number of increasing paths, which contradicts Theorem~\ref{prec=prec}.
\end{proof}

\begin{defn}
\label{defcfinitetype}
Let $F$ be a flipclass of finite type or a dihedral flipclass. We define $c(F)$ to be the number of increasing paths in $F$ (for any choice of reflection ordering).
\end{defn}

\begin{rmk}
  $ $

\begin{enumerate}
\item If  $F$ is a dihedral flipclass, then $c(F)=1$ by Proposition~\ref{oradipranzo}.
\item Let $F$ and $\bar{F}$ be two flipclasses of finite type. If $F \cong \bar{F}$ then $c(F)=c(\bar{F})$.
\item If $F$ is a flipclass of finite type and $F=F_1\ast F_2$, then also $F_1$ and $F_2$ are of finite type and $c(F)=c(F_1) \cdot c(F_2)$.
\end{enumerate}
\end{rmk}

\bigskip
We sum up the main results of this section in the following theorem.
\begin{thm}
  \label{thm:reduction to finite}
Let $F$ be an $h$-flipclass.
\begin{enumerate}
\item If $F$ is of finite type, then it is isomorphic to a flipclass $\bar{F}$ of a finite Coxeter group $\overline{W}$, and $c(F)=c(\bar{F})$.
\item  If $F$ is of Weyl type, then it is isomorphic to a flipclass $\bar{F}$ of a Weyl group $\overline{W}$,
and $c(F)=c(\bar{F})$.
\item  If $F$ is of symmetric type, then it is isomorphic to a flipclass $\bar{F}$ of a Coxeter group $\overline{W}$ 
that is isomorphic to a product of Weyl groups of type $A$, and $c(F)=c(\bar{F})$.
\end{enumerate}
Moreover, in all these cases, $\overline{W}$ may be  taken to be of rank not greater than $h$, and,  if $F$ is irreducible, then $\overline{W}$ may be  taken irreducible. 
\end{thm}
\begin{proof}
The proof is the same for the three cases. We prove the assertion with $W(F)$ as $\overline{W}$.

By Lemma~\ref{FeF'}, the flipclass $F$ is isomorphic to a flipclass $\bar{F}$ of $W(F)$ satisfying $c(F)=c(\bar{F})$. The rank of $W(F)$ does not exceed $h$, by Lemma~\ref{domanda}.
Furthermore, $W(F)$ is irreducible if $F$ is. 
\end{proof}

The following now readily follows.
\begin{cor}
\label{finiti}
Fix $h$ in $\mathbb N$. There are finitely many isomorphism classes of  $h$-flipclasses of finite type. 
\end{cor}

\begin{rmk}
Notice that $h$-flipclasses of finite Coxeter groups are infinite in number and come from intervals of any length.
\end{rmk}

\section{Flip Combinatorial invariance and Combinatorial invariance}
In this section, we extend the approach of \cite{EM} and we use the machinery developed in the previous sections to tackle the problem of the Combinatorial Invariance of the coefficients of the $\widetilde{R}$-polynomials.

Let $\mathcal F$ be a set of flipclasses (possibly in different groups).
We refer to the following predicate  as the \emph{Flip Combinatorial Invariance} for $\mathcal F$:
\begin{center}
for any two flipclasses $F,F'$ in $\mathcal F$, if $F\stackrel{\text{\tiny{cb}}}\cong F'$ then $c(F)=c(F')$.
\end{center}
\begin{rmk}
\label{sbadiglio}
$ $

\begin{enumerate}
\item
Flip Combinatorial Invariance does not hold for the set of 3-flipclasses of $H_3$ and $H_4$. The length 3 intervals that occur in a Weyl group are isomorphic to $2$-, $3$-, or $4$-crowns. Additionally, $H_3$ and $H_4$  contain also intervals isomorphic to $5$-crowns (see \cite{BB}, pg 52). The $3$-flipclasses of such intervals (which, as well as every $h$-flipclasses of a length $h$ interval, contain one increasing path each) are combinatorially isomorphic to some $3$-flipclasses (evidently of intervals of length greater than 3) containing two increasing paths and occuring both in Weyl groups and in $H_3$ and $H_4$. 
\item
\label{sb2}
 The fact that Flip Combinatorial Invariance be true for $\mathcal F_1$ and  $\mathcal F_2$ does not imply  that Flip Combinatorial Invariance be true for $\mathcal F_1 \cup \mathcal F_2$, unless one among $\mathcal F_1$ and $\mathcal F_2$ is closed under combinatorial isomorphism (such as, for example, the class of the dihedral flipclasses, see Lemma~\ref{nottefonda}(\ref{2})).
\end{enumerate}
\end{rmk}

Let $W$ be a Coxeter group, and $u,v \in W$. By Theorem~\ref{Dyertilde}, 
the coefficient $[q^h]\widetilde{R}_{u,v}(q)$ of $q^h$ in the $\widetilde{R}$-polynomial $\widetilde{R}_{u,v}(q)$ is the number of paths  from $u$ to $v$ of length $h$ that are increasing w.r.t. any reflection ordering. Thus,
$$[q^h]\widetilde{R}_{u,v}(q)= \sum_F c(F),$$
where the sum is over  all $h$-flipclasses $F$ of paths from $u$ to $v$. Since two isomorphic intervals share the same multiset of combinatorial isomorphism classes of $h$-flipclasses, we have the following result.
\begin{thm}
\label{FCICimpliesCIC}
Flip Combinatorial Invariance for $\mathcal F$ implies the Combinatorial Invariance Conjecture for the  intervals whose flipclasses are in  $\mathcal F$. 
\end{thm}
\begin{exm}
Consider the interval $[u,v]$ in Example~\ref{nottefonda2}. Then $c(F)=1$ for each of the four flipclasses $F$ of $[u,v]$. Indeed, $\widetilde{R}_{u,v}(q)=q^5+2q^3+q$. 
\end{exm}

Flipclasses and their combinatorial types are not easily described. Thence, the need of convenient invariants.
\begin{defn}
Let $\mathcal F$ be a set of flipclasses. A function having as domain $\mathcal F$ is a \emph{flip combinatorial invariant}  for $\mathcal F$ provided that it is constant on combinatorial isomorphism classes of $\mathcal F$.
Given two functions $f$ and $g$ having as domain $\mathcal F$, we say that $f$ \emph{refines} g provided that $g(F_1) = g(F_2)$ for all flipclasses $F_1$ and $F_2$ in $\mathcal F$ such that  $f(F_1) = f(F_2)$. 
\end{defn}

\begin{rmk}
$ $
\begin{enumerate}
\item Flip Combinatorial Invariance for  $\mathcal F$ is equivalent to saying that the function $c$ counting increasing paths is a flip combinatorial invariant  for $\mathcal F$.
\item If a flip combinatorial invariant $f$ refines $g$, then also $g$ is a flip combinatorial invariant.
\end{enumerate}
\end{rmk}

We adapt to our needs a definition from \cite[Section 1]{Chap}.
\begin{defn} Given a flipclass $F$, consider its time-support poset $TSP_F$.
  The \emph{valence polynomial} of $F$, denoted $\D(F)$,  is
  \[
    \D(F)
    =
    \sum_{(a,b)\in TSP_F\times TSP_F:\atop a\leq b}
    x^{\operatorname{in}(a,b)} y^{\operatorname{out}(a,b)},
  \]
  where
  \begin{align*}
    \operatorname{in}(a,b) & = |\{ c \in TSP_F | a\lhd c\le b \mbox{ or } a \le c \lhd b \} |, \\
    \operatorname{out}(a,b)& = |\{ c \in TSP_F | c\lhd a \mbox{ or } b \lhd c \} |,
  \end{align*}
 and $\lhd$ is used to indicate cover relations.
\end{defn}

\begin{rmk}
  \label{rmk:D mult}
  As our valence polynomial is the specialization of the polynomial defined in \cite{Chap} at $x = \overline x$ and $y = \overline y$, it inherits its properties.
  In particular, since the time-support poset is multiplicative (see Remark~\ref{moltimolti}), we have 
  \[
    \D(F_1 \ast F_2) = \D(F_1) \, \D(F_2).
  \]
  \end{rmk}

\begin{lem}
  \label{lem:same valence}
  If $F_1$ and $F_2$ are combinatorially isomorphic or combinatorially anti-isomorphic, then $\D(F_1) = \D(F_2)$.
\end{lem}

\begin{proof}
  If $F_1$ and $F_2$ are combinatorially isomorphic, they have isomorphic time-support posets and hence the same valence polynomial.
  If $F_1$ and $F_2$ are combinatorially anti-isomorphic, then their time-support posets are dual to each other, and this, due to the definition of the functions $\operatorname{in}$ and $\operatorname{out}$, gives rise to the same valence polynomial.
\end{proof}

\begin{rmk}
By Lemma~\ref{nottefonda}, a flipclass $F$ is dihedral if and only if one of the monomials of $\D(F)$ is $y^4$.
\end{rmk}

\begin{thm}
\label{flow}
The valence polynomial refines the function $c$ for:
\begin{enumerate}
\item the set of $h$-flipclasses of Weyl type with $h\in[6]$;
\item the set of $h$-flipclasses of symmetric type  with $h\in[7]$.
\end{enumerate}
\end{thm}

Not to break the flow of the section, we postpone the proof of Theorem~\ref{flow} to the next section.

\begin{cor}
\label{tttt}
Flip Combinatorial Invariance holds for
\begin{itemize}
\item dihedral flipclasses (of any length);
\item  $h$-flipclasses of Weyl type, for $h\leq 6$;
\item $h$-flipclasses of symmetric type, for $h\leq 7$.
\end{itemize}
\end{cor}
\begin{proof}
The first assertion follows by  Proposition~\ref{oradipranzo}. The second and the third assertions follow by  Theorem~\ref{flow}. 
\end{proof}

\begin{rmk}
Let $\mathcal D$ denote the set of dihedral flipclasses. If the Flip Combinatorial Invariance holds for a class $\mathcal F$, then it holds also for $\mathcal F \cup \mathcal D$, since $\mathcal D$ is closed under combinatorial isomorphism (see Remark~\ref{sbadiglio}(\ref{sb2})).
\end{rmk}

We believe that the following conjecture might be true. 
\begin{con}
\label{flipclasscongettura}
Flip Combinatorial Invariance holds for  flipclasses of Weyl type.
\end{con}

\subsection{Consequences for the Combinatorial Invariance Conjecture}
We explicitly state some consequences of the previous results.
\begin{thm}
\label{combinatoriainvarianza}
Let $[u_1,v_1]$ and   $[u_2,v_2]$ be two isomorphic intervals (of any length) in two Weyl groups. Then    
 $$\widetilde{R}_{u_1,v_1} \equiv  \widetilde{R}_{u_2,v_2}  \pmod {q^7}.$$
If, moreover, one restricts to Weyl groups of type $A$, then  
$$\widetilde{R}_{u_1,v_1} \equiv  \widetilde{R}_{u_2,v_2}  \pmod {q^8}.$$
\end{thm}
\begin{proof}
Follows from Corollary~\ref{tttt} by a similar argument to the proof of Theorem~\ref{FCICimpliesCIC}.
\end{proof}
\begin{thm}
The Combinatorial Invariance Conjecture holds for all intervals 
\begin{itemize}
\item up to length 10, for the set of Weyl groups,
\item
up to length 12, for the set of Weyl groups of type $A$.
\end{itemize}
\end{thm}
\begin{proof}
Recall that, given a polynomial $P$ in the variable $q$, we denote  the coefficient  of $q^s$ in  $P$ by  $[q^s]P(q)$.

Given $u$ and $v$ in an arbitrary Coxeter group with $u\leq v$, the polynomial  $\widetilde{R}_{u,v}(q)$ is  monic, has  degree $\ell(v) - \ell(u)$, and satisfies $[q^s]\widetilde{R}_{u,v} = 0$ whenever $s \not\equiv \ell(v) - \ell(u) \pmod2$. The coefficient $[q^{\ell(v)-\ell(u) -2}]\widetilde{R}_{u,v} $ is a combinatorial invariant by \cite[Corollary~1.3]{BGL}. Hence, Theorem~\ref{combinatoriainvarianza} implies the first assertion, as well as the fact that, for the set of type $A$ Weyl groups, the Combinatorial Invariance Conjecture holds for all intervals up to length 11. It is a general fact that, if the Combinatorial Invariance Conjecture holds for all intervals of a family of Coxeter groups up to length $k$, for a certain $k$ odd, then it holds also for intervals of length $k+1$.  This follows by the fact that the $\widetilde{R}$-polynomials of an interval $[u,v]$ satisfy
$\sum_{z\in[u,v]}(-1)^{\ell(z)-\ell(u)}\widetilde{R}_{u,z}\widetilde{R}_{z,v}=0$, for all $u\neq v$.
\end{proof}
The Combinatorial Invariance Conjecture was previously known to hold for intervals  up to length $4$ for all Coxeter groups (see \cite [Chapter 5,  Exercises 5, 6, 7 and 8]{BB}),  up to length $6$ for the set of Weyl groups of type $B$, up to length $6$ for the set of Weyl groups of type $D$, and up to length $8$ for the set of Weyl groups of type $A$ (see \cite{Inc}). In the recent paper \cite{BGL}, it is established for intervals up to length 6 for arbitrary Coxeter groups.

\section{Proof of Theorem~\ref{flow}}

We have to prove that, given two Weyl type $h$-flipclasses $F$ and $F'$ with $h\in[6]$, or two symmetric type $h$-flipclasses $F$ and $F'$ with $h\in[7]$, if they have the same valence polynomial,  then they have the same number of increasing paths, i.e.,  $ \D(F) = \D(F')$ implies $c(F)=c(F')$. Theorem~\ref{thm:reduction to finite} reduces this analysis to a finite number of cases. Indeed, we may suppose that $F$ and $F'$ are flipclasses of a rank $h$ Coxeter group that is a Weyl group in the first case, or a product of Weyl groups of type $A$ in the second case.

The proof of both claims is performed simultaneously by direct inspection with the help of a computer algebra software \cite{Sage}.
In order to reduce the amount of computations to be performed to a tractable amount, we apply several reductions that we will discuss after describing the strategy of the proof.

Throughout this section, $\mathcal{F}$ will denote the union of the set of all  $h$-flipclasses, with $h\in[6]$, in a Weyl group of rank at most $6$  and the set of all  $h$-flipclasses, with $h\in[7]$, in a rank at most $7$ group which is a product of type $A$ Weyl groups.
Define the sets
\[
  \mathfrak{D} = \left\{ \D(F) \,|\, F\in\mathcal{F} \right\} \qquad \text{and } \qquad \mathfrak{C} = \left\{ c(F) \,|\, F\in\mathcal{F} \right\}. 
\]

Theorem~\ref{flow} claims the existence of a map $\gamma:\mathfrak{D} \rightarrow \mathfrak{C}$ such that $\gamma(\D(F)) = c(F)$ for any $F\in\mathcal{F}$.
We construct $\gamma$ from a similar map satisfying slightly more restrictive assumptions.

Specifically, let $\overline{\mathfrak{D}}$ be the set of all irreducible factors of polynomials in $\mathfrak{D}$. 
We verify that there is a map $\overline\gamma:\overline{\mathfrak{D}}\rightarrow \mathbb{R}$ such that, if $\D(F) = d_1\cdots d_\ell$ is a factorization in irreducibles, then $\overline\gamma(d_1)\cdots \overline\gamma(d_\ell) = c(F)$.
Clearly such a $\overline\gamma$ yields the desired map $\gamma$ by multiplication.
On the other hand, the requirement that $\overline\gamma$ is defined on irreducible factors imposes extra conditions since the valence polynomial $\D(F)$ might factorize even when $F$ is irreducible (cf. Remark~\ref{rmk:virtual}).

Let $\mathcal{F}'$ be the subset of $\mathcal{F}$ consisting only of the $h$-flipclasses contained in an irreducible group. 
In view of Remark~\ref{rmk:D mult} and Corollary~\ref{cor:unique factorization}, the sets of all irreducible factors of polynomials in $\left\{ \D(F) \,|\, F\in\mathcal{F}' \right\}$ coincides with $\overline{\mathfrak{D}}$.
It suffices therefore to consider only flipclasses in $\mathcal{F}'$.

Denote by $\mathcal{F}'(h)$ the subset of $\mathcal{F}'$ consisting only of $h$-flipclasses.
By an abuse of notation, set $\overline\gamma(d_1 d_2)=\overline\gamma(d_1)\overline\gamma(d_2)$ when both $\overline\gamma(d_1)$, and $\overline\gamma(d_2)$ are defined.
Moreover set $\overline\gamma(1)=1$.
We define $\overline\gamma$ and simultaneously check its consistency by the following algorithm.
\begin{algorithmic}
  \For{$h\in \mathbb{N}$}
		\For{$F\in\mathcal{F}'(h)$}
      \State $g \gets$ largest factor of $\D(F)$ on which $\overline\gamma$ is defined
      \If{$g = \D(F)$}
        \State check that $\overline\gamma(g) = c(F)$ or fail
      \Else
        \State $k \gets $ number of irreducible factors of $\D(F)/g$ counted with multiplicity
        \For{$d$ irreducible factor of $\D(F)/g$}
        \State $\overline\gamma(d) \gets (c(F)/\overline\gamma(g))^{1/k}$
        \label{alg-step: roots} 
        \EndFor
      \EndIf
		\EndFor	
	\EndFor
\end{algorithmic}

In order to run successfully the above algorithm, we need a way to construct all the $h$-flipclasses in an irreducible group $W$.
Since, by Proposition~\ref{greedy}, the integer $c(F)$ is at least $1$, this can be done by listing, for all elements $u\in W$, all the increasing paths of length $h$ starting at $u$ and computing, for each of them, the corresponding flipclass.

Building flipclasses is by far the most computationally intensive part of the proof of Theorem~\ref{flow}.
The first observation one can make in order to reduce the number of cases to be checked is that certain Weyl groups appear naturally as subgroups of other Weyl groups.
It suffices therefore to consider only $h$-flipclasses in the groups listed in Table~\ref{table:numerology}.

\begin{table}[h!]
\caption{
  Some statistics about the number of $h$-flipclasses and the number of associated valence polynomials for irreducible Weyl groups.
}
  \label{table:numerology}
\centering
\begin{tabular}{ c || c | r | r | r | r | r  }
  \
  {\small Type} & $h$ & {\small \begin{tabular}{c}elements\\ to\\ check\end{tabular}} & {\small flipclasses} &  {\small \begin{tabular}{c}valence\\ polynomials\end{tabular}} & {\small \begin{tabular}{c}irreducible\\ valence\\ polynomials\end{tabular}} & {\small \begin{tabular}{c}new\\ irreducible\\ valence\\ polynomials\end{tabular}} \\
\hline
\hline
$A_1$ & 1 & 1          & 1                     & 1             & 1              & 1                   \\
\hline
$A_2$ & 2 & 1          & 1                     & 1             & 0              & 0                   \\
$B_2$ & 2 & 3          & 8                     & 1             & 0              & 0                   \\
$G_2$ & 2 & 5          & 25                    & 1             & 0              & 0                   \\
\hline
$A_3$ & 3 & 3          & 15                    & 4             & 3              & 3                   \\
$B_3$ & 3 & 16         & 216                   & 8             & 7              & 4                   \\
\hline
$A_4$ & 4 & 16         & 363                   & 11            & 7              & 7                   \\
$B_4$ & 4 & 125        & 11987                 & 206           & 198            & 191                 \\
$D_4$ & 4 & 53         & 2283                  & 19            & 15             & 0                   \\
$F_4$ & 4 & 437        & 144281                & 1765          & 1757           & 1585                \\
\hline
$A_5$ & 5 & 92         & 11343                 & 100           & 89             & 89                  \\
$B_5$ & 5 & 1255       & 1085742               & 17213         & 17007          & 16918               \\
$D_5$ & 5 & 285        & 102724                & 1203          & 1184           & 402                 \\
$F_4$ & 5 & 350        & 114018                & 8421          & 8317           & 7726                \\
\hline
$A_6$ & 6 & 488        & 425442                & 2362          & 2256           & 2256                \\
$B_6$ & 6 & 15232      & 144074849             & 1838321       & 1821079        & 1818824             \\
$D_6$ & 6 & 6896       & 31903124              & 174545        & 173336         & 103645              \\
$E_6$ & 6 & 8720       & 95718946              & 1112864       & 1111655        & 1007217             \\
$F_4$ & 6 & 350        & 99171                 & 10053         & 9727           & 9029                \\
\hline
$A_7$ & 7 & 4072       & 25935215              & 95296         & 92913          & 92913               \\
\end{tabular}
\end{table}

Using the anti-isomorphism $(x\mapsto xw_0, t\mapsto t)$ and Lemma~\ref{lem:same valence}, we can further restrict our attention to $h$-flipclasses from $u$ to $v$ such that 
\[
  \ell(u) \leq \frac{\ell(w_0)-h}{2}
  \qquad
  \mbox{and}
  \qquad
  \ell(u) \leq \ell(w_0) - \ell(v).
\]

Finally, recall that in any Weyl group a reduced word for a \emph{Coxeter element} $c=s_1\cdots s_n$ (i.e. any total order among the Coxeter generators) induces a total order on the group.
Namely, given $u\in W$ the \emph{$c$-sorting word} $c(u)$ of $u$ is the lexicographically first subword of the infinite word $c^\infty=s_1\cdots s_n s_1 \cdots$ that is a reduced word for $u$.
We say that $u\prec_c v$ whenever $c(u)$ precedes $c(v)$ lexicographically as subwords of $c^\infty$
(C.f. \cite{Rea1, Rea2} for further details). 

Fixing one such order for any irreducible Weyl group, using Lemma~\ref{lem:same valence} and the isomorphism $(x\mapsto w_0xw_0, t\mapsto w_0tw_0)$, we can further restrict our study to flipclasses from $u$ to $v$ such that either $u\prec_c w_0 u w_0$ or $ u = w_0 u w_0$ and $v\preceq_c w_0 v w_0$.
We note that the choice of this total order is immaterial: its only feature we use is that elements can be quickly compared without listing first the entire group; any other total order on $W$ with the same feature would work.

The number of remaining cases to be checked after all these reductions are listed in Table~\ref{table:numerology}.
There, for each group, we indicate how many elements should be considered as starting point for flipclasses and how many flipclasses satisfying the above conditions there are.
We then list how many different valence polynomials these flipclasses correspond to, how many of these are irreducible and how many of these irreducibles did not appear already in a group listed in a preceding row.

We conclude this section with two experimental observations.

\begin{rmk}
  Even though the last step in our algorithm may in theory define $\overline\gamma(d)$ not to be an integer, this never happens in the cases we considered.
  Specifically, parsing the groups in the order given by Table~\ref{table:numerology}, whenever it is not trivial, $\D(F)/g$ is always irreducible and therefore has $k=1$.
  Moreover in these cases $\overline\gamma(g)$ always divides $c(F)$.
  Therefore $\overline\gamma$ only assumes integer values on $\overline{\mathfrak{D}}$.
\end{rmk}

\begin{rmk}
  \label{rmk:virtual}
  As noted above, it is possible that $\D(F)$ is reducible even though the flipclass $F$ is irreducible. 
  On the one hand, it is possible that a valence polynomial factors into polynomials that are not themselves valence polynomials of some other flipclasses.
  However, this phenomenon seems to be rare.
  Indeed, in the computations we performed, there is a unique flipclass up to combinatorial isomorphisms and anti-isomorphisms for which this is the case.
  It is an irreducible 6-flipclass which appears only in $B_6$ and  has  valence polynomial
  \[
    (x^2 + 2 y) \left(\begin{array}{c}
        x^{20} + 4 y x^{14} + 2 y x^{13} + 4 y^{2} x^{11} + 2 y x^{12} + 5 y^{2} x^{10} + 2 y^{3} x^{8} + 2 y x^{10} + 2 y^{10} + \\
        2 y^{6} x^{4} +  4 y^{2} x^{8} + 10 y^{7} x^{2} + 2 y^{5} x^{4} + 6 y x^{8} + 8 y^{8} + 6 y^{6} x^{2} + 2 y^{4} x^{4} + 82 y^{2} x^{6} + \\
        4 y^{7} + 6 y^{5} x^{2} + 198 y^{3} x^{4} + 4 y x^{6} + 4 y^{6} + 170 y^{4} x^{2} + 6 y^{2} x^{4} + 48 y^{5} + 8 y^{3} x^{2} + 6 y^{4}
      \end{array} \right).
  \]
  In particular, while the first factor is the valence polynomial of the unique flipclass in $A_1$, the second factor is not the valence polynomial of any 5-flipclass.

  On the other hand, it may also be possible that $\D(F)$ factors as the product of other valence polynomials coming from smaller flipclasses but that $F$ is irreducible.
  Checking if this latter situation arises is much more involved as one has to try to factorize flipclasses; we did not perform such a computation.
  This situation may in theory lead to inconsistencies in the definition of $\gamma$ but we did not encounter any.
\end{rmk}

\bigskip
{\bf Acknowledgments:} 
F.E. and M.M. would like to thank Grant T. Barkley, Francesco Brenti, Matthew Dyer, and Christian Gaetz for stimulating discussions we had during the Workshop \lq\lq Bruhat order: recent developments and open problems\rq\rq, held at the University of  Bologna from 15 April 2024 to 19 April 2024. S.S. would like to thank Vincent Pilaud for interesting insights. 
The authors would like to thank Frédéric Chapoton for useful information about polynomial invariants of posets.
The authors are grateful to Christophe Hohlweg for helpful discussions about reflection subgroups of infinite Coxeter groups. F.E. acknowledges PRIN 2022S8SSW2, \lq\lq Algebraic and geometric aspects of Lie theory\rq\rq.
M.M. acknowledges  PRIN 2022A7L229, \lq\lq ALgebraic and TOPological combinatorics\rq\rq. 
S.S. acknowledges  PRIN 20223FEA2E, \lq\lq Cluster algebras and Poisson Lie groups\rq\rq.  
 F.E. wishes to thank the Universit\`a Politecnica delle Marche for hospitality.
The authors are members of the INDAM group GNSAGA.
Most of the computations in this paper have been performed on the Linux HPC cluster Caliban of the High Performance Parallel Computing Laboratory of the Department of Information Engineering, Computer Science and Mathematics (DISIM) at the University of L’Aquila.

\end{document}